\documentclass[11pt]{article}
\usepackage{amsfonts,amssymb,amsmath,amsthm,epsfig,euscript,epstopdf}
\usepackage[hidelinks]{hyperref}
\usepackage[noadjust]{cite}
\newtheorem{theorem}{Theorem}
\newtheorem{lemma}[theorem]{Lemma}

\newtheorem*{thm}{Theorem}

\long\def\symbolfootnote[#1]#2{\begingroup
\def\thefootnote{\fnsymbol{footnote}}\footnote[#1]{#2}\endgroup}

\newcommand{\la}{\lambda}
\newcommand{\La}{\Lambda}

\newcommand{\des}{\mathrm{des}}

\newcommand{\LRmin}[1]{\mathrm{LRmin}(#1)}

\newcommand{\nth}[1][n]{{#1}^{\mathrm{th}}}

\newcommand{\sg}{\sigma}

\newcommand{\cref}[1]{Corollary \ref{corollary:#1}}

\newcommand{\Floor}[1][n/2]{\left \lfloor #1 \right \rfloor}

\newcommand{\red}{\mathrm{red}}

\newcommand{\sgn}[1]{\mathrm{sgn}(#1)}

\newcommand{\tmch}{\text{$\tau$-$\mathrm{mch}$}}

\newcommand{\Gmch}{\text{$\Gamma$-$\mathrm{mch}$}}

\newcommand{\fig}[2]{\begin{figure}[ht]
\centerline{\scalebox{.66}{\epsfig{file=#1.eps}}}
\caption{#2}
\label{fig:#1}
\end{figure}}

\title{Generating functions for permutations which avoid consecutive patterns with multiple descents.}

\author{
Quang T. Bach \\
\small Department of Mathematics\\[-0.8ex]
\small University of California, San Diego\\[-0.8ex]
\small La Jolla, CA 92093-0112. USA\\[-0.8ex]
\small \texttt{qtbach@ucsd.edu}
\and
\and
Jeffrey B. Remmel \\
\small Department of Mathematics\\[-0.8ex]
\small University of California, San Diego\\[-0.8ex]
\small La Jolla, CA 92093-0112. USA\\[-0.8ex]
\small \texttt{remmel@math.ucsd.edu}
\and
}

\date{\small Submitted: Date 1;  Accepted: Date 2;
 Published: Date 3.\\
\small MR Subject Classifications: 05A15, 05E05 \\
keywords: pattern avoidance, consecutive pattern, permutation, pattern match, descent, left to right minimum, symmetric polynomial, exponential generating function}

\begin{document}

\maketitle

\section{Introduction}

Let $S_n$ denote the group all permutations of $n$. That is,
$S_n$ is the set of all one-to-one maps $\sg:\{1, \ldots, n\} \rightarrow 
\{1, \ldots, n\}$ under composition.  If $\sg = \sg_1 \ldots \sg_n \in S_n$, 
then we let $Des(\sg) = \{i: \sg_i >\sg_{i+1}\}$ and 
$\des(\sg) =|Des(\sg)|$. We say that $\sg_j$ is 
left-to-right minima of $\sg$ if $\sg_i > \sg_j$ for all $i < j$. 
For example the \emph{left-to-right minima} of $\sigma=938471625$ are 
$9$, $3$ and $1$. 
Given a sequence $\tau = \tau_1 \cdots \tau_n$ of distinct positive integers,
we define the \emph{reduction} of $\tau$, $\red(\tau)$, to be the permutation 
of $S_n$ that results by replacing the $i$-th smallest element 
of $\tau$ by $i$. For example 
$\red(53962) = 32541$. If $\Gamma$ is a set of permutations, 
we say that a permutation $\sg = \sg_1 \ldots \sg_n \in S_n$ has 
a $\Gamma$-match starting at position $i$ if there is a $i < j$ such 
that $\red(\sg_i \sg_{i+1} \ldots \sg_j) \in \Gamma$. 
We let $\Gmch(\sg)$ denote the number of $\Gamma$-matches in $\sg$.
We let $\mathcal{NM}_n(\Gamma)$ be the set of $\sg \in S_n$ such 
that $\Gmch(\sg) =0$. 

The main goal of this paper is to study generating functions 
of the form 
\begin{equation}\label{eq:I1}
\mbox{NM}_{\Gamma}(t,x,y)=\sum_{n \geq 0} \frac{t^n}{n!} 
\mbox{NM}_{\Gamma,n}(x,y)
\end{equation}
where $\displaystyle \mbox{NM}_{\Gamma,n}(x,y) =\sum_{\sg \in \mathcal{NM}_n(\Gamma)}x^{\LRmin{\sg}}y^{1+\des(\sg)}.$  
In the special case where $\Gamma = \{\tau\}$ is a set with a 
single permutation $\tau$, we shall write $\tmch(\sg)$ for 
$\Gmch(\sg)$, $\mbox{NM}_{\tau}(t,x,y)$ for $\mbox{NM}_{\Gamma}(t,x,y)$,  
and  $\mbox{NM}_{\tau,n}(x,y)$ for $\mbox{NM}_{\Gamma,n}(x,y)$.

There is a considerable literature on the 
generating function $\mbox{NM}_{\Gamma}(t,1,1)$ of 
permutations that consecutively avoid a pattern 
or set of patterns. See for example, \cite{AAM,B1,B2,DK,DR,EN,EN2,EKP,GJ, Kit1,Kitbook,MenRem}. For the most part, these papers do not  
consider generating functions of the form 
$\mbox{NM}{\tau}(t,1,y)$ or $\mbox{NM}{\tau}(t,x,y)$. An exception 
is the work on enumeration schemes of Baxter \cite{B1,B2} who gave 
general methods to enumerate pattern avoiding vincular patterns according 
to various permutations statistics. Our approach is 
to use the reciprocity method of Jones and Remmel.

Jones and Remmel \cite{JR,JR2,JR3} 
developed what they called the reciprocity method 
to compute the generating function 
$\mbox{NM}{\tau}(t,x,y)$ for certain families of 
permutations $\tau$ such that $\tau$ starts with 
1 and $\des(\tau) =1$.

The basic idea of their approach 
is as follows. First it follows from results in \cite{JR} that if 
all the permutations in $\Gamma$ start with 1, then we 
can write $\mbox{NM}_{\Gamma}(t,x,y)$ in the form 
\begin{equation}\label{eq:I2}
\mbox{NM}_{\Gamma}(t,x,y) = \left( \frac{1}{U_{\Gamma}(t,y)}\right)^x
\end{equation}
 where $\displaystyle 
U_{\Gamma}(t,y) = \sum_{n\geq 0}U_{\Gamma,n}(y) \frac{t^n}{n!}$.  
Next one writes 
\begin{equation}\label{eq:I3}
U_{\tau}(t,y) = 
\frac{1}{1+\sum_{n \geq 1} \mbox{NM}_{\tau,n}(1,y) \frac{t^n}{n!}}.
\end{equation} One can then use the homomorphism method 
to give a combinatorial   
interpretation of the right-hand side of (\ref{eq:I3}) which can 
be used to find simple recursions for  
the coefficients $U_{\tau,n}(y)$. 
The homomorphism method derives generating functions for 
various permutation statistics by 
applying a ring homomorphism defined on the 
ring of symmetric functions \begin{math}\Lambda\end{math}  
in infinitely many variables \begin{math}x_1,x_2, \ldots \end{math} 
to simple symmetric function identities such as 
\begin{equation}\label{conclusion2}
H(t) = 1/E(-t)
\end{equation}
where $H(t)$ and $E(t)$ are the generating functions for the homogeneous and elementary 
symmetric functions, respectively. That is, 
\begin{equation}\label{genfns}
H(t) = \sum_{n\geq 0} h_n t^n = \prod_{i\geq 1} \frac{1}{1-x_it} 
\ \mbox{and} \ E(t) = \sum_{n\geq 0} e_n t^n = \prod_{i\geq 1} 1+x_it.
\end{equation}
In their case, Jones and Remmel
 defined a homomorphism \begin{math}\theta_{\tau}\end{math} on 
\begin{math}\Lambda\end{math} by setting 
\begin{displaymath}\theta_{\tau}(e_n) = \frac{(-1)^n}{n!} \mbox{NM}_{\tau,n}(1,y).\end{displaymath}
Then 
\begin{displaymath}\theta_{\tau}(E(-t)) = {\sum_{n\geq 0} \mbox{NM}_{\tau,n}(1,y) \frac{t^n}{n!}} = \frac{1}{U_\tau(t,y)}.\end{displaymath}
Hence 
$$U_\tau(t,y) = \frac{1}{\theta_{\tau}(E(-t))} = \theta_{\tau}(H(t))$$
which implies that 
\begin{equation}\label{eq:combhn}
n!\theta_{\tau}(h_n) = U_{\tau,n}(y).
\end{equation}
Thus if we can compute $n!\theta_{\tau}(h_n)$ for all $n \geq 1$, then we can 
compute the polynomials $U_{\tau,n}(y)$ and the generating function 
$U_{\tau}(t,y)$, which in turn allows us to compute 
the generating function $\mbox{NM}_{\tau}(t,x,y)$. 
Jones and Remmel \cite{JR2,JR3} showed that one can interpret 
$n!\theta_{\tau}(h_n)$ as a certain signed sum of weights of filled labeled brick tabloids when $\tau$ starts with 1 
and $\des(\tau)=1$. They then defined a weight-preserving sign-reversing 
involution $I$ on the set of such filled labeled brick tabloids which 
allowed them to give a relatively simple combinatorial interpretation 
for $n!\theta_{\tau}(n_n)$. They also showed how such a  
combinatorial interpretation allowed them to prove  
that the polynomials $U_{\tau,n}(y)$ satisfy simple recursions
for certain families of such permutations $\tau$. 

For example, in \cite{JR2}, Jones and Remmel  studied  the generating functions $\mbox{NM}_{\tau}(t,x,y)$ for permutations $\tau$ of the form $\tau = 1324\cdots p$ where $p \geq 4$.  Using the reciprocity method, they proved that $U_{1324,1}(y)=-y$ and for $n \geq 2$, 
\begin{equation}\label{1324}
U_{1324,n}(y) = (1-y)U_{1324,n-1}(y)+ \sum_{k=2}^{\lfloor n/2 \rfloor} 
(-y)^{k-1} C_{k-1} U_{1324,n-2k+1}(y)
\end{equation} where $C_k = \frac{1}{k+1}\binom{2k}{k}$ is the $k$-th Catalan number. They also proved that for  any $p \geq 5$,  $U_{1324 \cdots p,n}(y) =-y$ and for $n \geq 2$, 
\begin{equation}\label{1324p}
U_{1324\cdots p,n}(y)=(1-y)U_{1324\cdots p,n-1}(y)+\sum_{k=2}^{\lfloor\frac{n-2}{p-2}\rfloor+1}(-y)^{k-1}U_{1324\cdots p,n-((k-1)(p-2)+1)}(y).
\end{equation}

Bach and Remmel \cite{BR} extended this reciprocity method to 
study the polynomials 
$U_{\Gamma,n}(y)$ in the case where $\Gamma$ is a set of permutations 
such that for all $\tau \in \Gamma$, $\tau$ starts with 1 and 
$\des(\tau) \leq 1$. For example, 
suppose that $k_1, k_2 \geq 2$, $p = k_1 + k_2$, 
and  
$$\Gamma_{k_1,k_2} = \{\sigma \in S_p: \sigma_1=1, \sigma_{k_1+1}=2, \sigma_1 < \sigma_2< \cdots<\sigma_{k_1}~ \&~\sigma_{k_1+1} < \sigma_{k_1+2}< \cdots<\sigma_{p} \}.$$ 
That is, $\Gamma_{k_1,k_2}$ consists of all permutations $\sg$ of length $p$ where 1 is in position 1, 2 is in position $k_1+1$, and $\sg$ consists of two increasing sequences, one starting at 1 and the other starting at 2. In \cite{BR}, 
we proved that for $\Gamma = \Gamma_{k_1,k_2}$, 
$U_{\Gamma,1}(y)=-y$, and for $n \geq 2,$
\begin{align*}
\displaystyle   U_{\Gamma,n}(y) &= (1-y)U_{\Gamma,n-1}(y) -y\binom{n-2}{k_1-1}\left( U_{\Gamma,n-M}(y) +y\sum_{i=1}^{m-1}U_{\Gamma,n-M-i}(y) \right) 
\end{align*} where $m = \min\{k_1, k_2\}$, and $M = \max\{k_1,k_2\}$. 

Furthermore, in \cite{BR}, we investigated a new phenomenon that arises when we add the identity permutation $12 \ldots k$ to the family $\Gamma$. 
 For example, if $\Gamma = \{1324,123\}$, then we proved 
that  $U_{\Gamma,1}(y)=-y$, and for $n \geq 2,$
\begin{equation}\label{eq:1324,123}
U_{\Gamma,n}(y) = -yU_{\Gamma,n-1}(y) -yU_{\Gamma,n-2}(y)
+ \sum_{k=2}^{\Floor }(-y)^{k}C_{k-1}U_{\Gamma, n-2k}(y).
\end{equation} 
When $\Gamma = \{1324\dots p,123\dots p-1\}$ where $p \geq 5,$ then 
we proved that $U_{\Gamma,1}(y)=-y$, and for $n \geq 2,$
\begin{equation}\label{1324p,12p}
U_{\Gamma,n}(y) = \sum_{k=1}^{p-2}(-y)U_{\Gamma,n-k}(y) + \sum_{k=1}^{p-2}\sum_{m=2}^{\lfloor\frac{n-k}{p-2}\rfloor }(-y)^{m}U_{\Gamma, n-k-(m-1)(p-2)}(y). 
\end{equation} 
While on the surface, 
the recursions (\ref{eq:1324,123}) and (\ref{1324p,12p}) do not 
seem to be simpler than the corresponding recursions 
(\ref{1324}) and (\ref{1324p}), they are easier to analyze because 
adding an identity permutation $12 \ldots k$ to $\Gamma$ ensures 
that all the bricks in the filled brick tabloids used to interpret 
$n!\theta_{\tau}(h_n)$ have length less than $k$. For 
example, we were able to prove the following explicit formula 
for the polynomials $U_{\{1324,123\},n}(y)$. 

\begin{theorem} \label{1324,123}
Let $\Gamma = \{1324,123\}$. Then for all 
$n \geq 0$, 
\begin{equation}\label{u2n}
U_{\Gamma,2n}(y) = 
\sum_{k=0}^n \frac{(2k+1)\binom{2n}{n-k}}{n+k+1}(-y)^{n+k+1}
\end{equation}
and 
\begin{equation}\label{u2n+1}
U_{\Gamma,2n+1}(y) = 
\sum_{k=0}^n \frac{2(k+1)\binom{2n+1}{n-k}}{n+k+2}(-y)^{n+k}.
\end{equation}
\end{theorem}

Another example in \cite{BR} where we could find 
an explicit formula is the following. Let $\Gamma_{k_1,k_2,s} = \Gamma_{k_1,k_2} \cup \{1 \cdots s(s+1)\}$ for some $s \geq \max(k_1,k_2)$. Bach and Remmel showed that $U_{\Gamma_{2,2,s},1}(y)=-y$, and for $n \geq 2,$
\begin{align} \label{Gamma22s}
U_{\Gamma_{2,2,s},n}(y)&= -yU_{\Gamma_{2,2,s},n-1}(y) - \nonumber\\
&\ \ \ \sum_{k=0}^{s-2} \left((n-k-1) yU_{\Gamma_{2,2,s},n-k-2}(y)+(n-k-2)
y^2 U_{\Gamma_{2,2,s},n-k-3}(y)\right).
\end{align} 
Using these recursions, we proved that 
\begin{align*}
U_{\Gamma_{2,2,2},2n}(y) = & \sum_{i=0}^n (2n-1)\downarrow \downarrow_{n-i} (-y)^{n+i} \ \mbox{~~and} \\
U_{\Gamma_{2,2,2},2n+1}(y) =& \sum_{i=0}^n (2n) \downarrow \downarrow_{n-i} (-y)^{n+1+i}
\end{align*} where for any $x$, $(x)\downarrow \downarrow_{0} =1$ and $(x)\downarrow \downarrow_{k} =x(x-2)(x-4) \cdots (x-2k -2)$ for $k \geq 1$.

The two assumptions on $\Gamma$ that allow the reciprocity  
method to work are that (A) all $\tau$ in $\Gamma$ start with 
1 and (B) all $\tau$ in $\Gamma$ have at most 
one descent. First, assumption (A) ensures that we 
can write $\mbox{NM}_{\Gamma}(x,y,t)$ in the form (\ref{eq:I2}). 
Second, assumption (B) ensures that the involution $I$ used to 
simplify the weighted sum over all filled, labeled brick tabloids 
that equals $n!\theta_{\tau}(h_n)$ is actually an involution and to 
ensure that the elements in any brick of a filled, labeled 
brick tabloids which is a fixed point of $I$ 
must be increasing. Finally, (A) is used again to ensure that the minimal 
elements in bricks of any fixed point of $I$ are 
increasing when read from left to right.

The main goal of this paper is to study how we can 
apply the reciprocity method in the case 
where we no longer insist that all the 
$\tau \in \Gamma$ have at most one descent. We shall 
show  that we can modify the definition of the involution  
used by Jones and Remmel \cite{JR2,JR3} and Bach and Remmel \cite{BR} 
to simplify the weighted sum over all filled, labeled brick tabloids 
that equals $n!\theta_{\tau}(h_n)$.  However, the set of fixed points 
in such cases will be more complicated than 
in the case where $\Gamma$ contains only permutations 
with at most one descent in that it will no longer be the 
case that for fixed points of the involution, 
the fillings will be increasing in bricks and 
the minimal elements of the brick increase, reading from left 
to right.  Nevertheless, we shall show that there 
still are a number of 
cases where we can successfully analyze the fixed points 
to prove that the polynomials $U_{\Gamma,n}(y)$ satisfy some simple recursions.

In this paper, we shall prove three main theorems. 
That is, we will compute the generating functions 
$\mbox{NM}_{\Gamma}(t,x,y)$ when 
$\Gamma = \{14253,15243\}$, $\Gamma = \{142536\}$, 
and when $\Gamma = \{\tau_a\}$ for any 
$a \geq 2$ where $\tau_a \in S_{2a}$ is 
the permutation such that 
$\tau_1 \tau_3 \ldots \tau_{2a-1} = 12 \ldots a$ and 
$\tau_2 \tau_4 \ldots \tau_{2a}= (2a) (2a-1) \ldots (a+1)$. 
In each case, the permutations have at least two descents.  
In \cite{BR2}. we studied the generating functions of the form 
$\mbox{NM}_\tau(t,x,y)$ where $\tau$ is a minimal overlapping 
permutation that starts with 1.  Here $\tau \in S_j$ is a minimal 
overlapping permutation if the smallest $j$ such that 
there exists an $\sg \in S_n$ such that $\tmch(\sg) =2$ is $2j-1$. 
This means that any two consecutive $\tau$-matches can share 
at most one letter. When $\tau$ is a minimally 
overlapping permutations, the recursions for $U_{\tau,n}(y)$ are generally much 
simpler than the ones considered in this paper because in each case 
we are dealing with permutations which are not minimally overlapping.

The main results of this paper are the following theorems. 

\begin {theorem} \label{thm:1-2-3}
Let $\Gamma = \{14253,15243\}$. Then 
$$\mbox{NM}_\Gamma(t,x,y)=\left(\frac1{U_\Gamma(t,y)}\right)^x \text{ where }U_\Gamma(t,y)=1+\sum_{n\geq1}U_{\Gamma,n}(y)\frac{t^n}{n!},$$ with $U_{\Gamma,1}(y)=-y$, and for $n \geq 2,$
\begin{align*}
\displaystyle  U_{\Gamma,n}(y) & = (1-y)U_{\Gamma,n-1}(y)  -y^2(n-3)\left(U_{\Gamma, n-4}(y) +(1-y)(n-5)U_{\Gamma, n-5}(y)  \right)\\
& \qquad \qquad \qquad \qquad \qquad -y^3(n-3)(n-5)(n-6)U_{\Gamma,n-6}(y).
\end{align*}
\end{theorem} 

Let $C_n = \frac{1}{n+1}\binom{2n}{n}$ be the $n$-th Catalan number. 
Let $M_n$ be the $n\times n$ matrix whose elements on 
the main diagonal equals $C_2$, whose elements on $j$-th diagonal above 
the main diagonal are $C_{3j+2}$, whose elements on the sub-diagonal 
are $-1$, and whose elements in diagonal below the sub-diagonal are 0. 
Thus, 
\begin{equation*}  M_k = \begin{vmatrix}
C_2 & C_5 & C_8 & C_{11} & \cdots & C_{3k-4} & C_{3k-1} \\ 
-1 & C_2 & C_5 & C_8 & \cdots & C_{3k-7} & C_{3k-4} \\ 
0 & -1 & C_2 & C_5 & \cdots & C_{3k-10} & C_{3k-7} \\ 
0 & 0 & -1 & C_2 & \cdots & C_{3k-13} & C_{3k-10} \\ 
\vdots & \vdots & \vdots & \vdots &  & \vdots & \vdots \\ 
0 & 0 & 0 & 0 & \cdots & C_2 & C_5 \\ 
0 & 0 & 0 & 0 & \cdots & -1 & C_2
\end{vmatrix}.
\end{equation*}
Let $P_k$ be the matrix obtained from $M_k$ by replacing each 
$C_m$ in the last column by $C_{m-1}$. Thus, 
\begin{equation*} P_k = \begin{vmatrix}
C_2 & C_5 & C_8 & C_{11} & \cdots & C_{3k-4} & C_{3k-2} \\ 
-1 & C_2 & C_5 & C_8 & \cdots & C_{3k-7} & C_{3k-5} \\ 
0 & -1 & C_2 & C_5 & \cdots & C_{3k-10} & C_{3k-8} \\ 
0 & 0 & -1 & C_2 & \cdots & C_{3k-13} & C_{3k-11} \\ 
\vdots & \vdots & \vdots & \vdots &  & \vdots & \vdots \\ 
0 & 0 & 0 & 0 & \cdots & C_2 & C_4 \\ 
0 & 0 & 0 & 0 & \cdots & -1 & C_1
\end{vmatrix}.
\end{equation*}

\begin {theorem} \label{thm:142536}
Let $\tau = 142536$. Then 
$$\mbox{NM}_\tau(t,x,y)=\left(\frac1{U_\tau(t,y)}\right)^x \text{ where }U_\tau(t,y)=1+\sum_{n\geq1}U_{\tau,n}(y)\frac{t^n}{n!},$$ with $U_{\tau,1}(y)=-y$, and for $n \geq 2,$
\begin{align*}
\displaystyle U_{\tau,n}(y) = &~~~  (1-y)U_{\Gamma,n-1}(y) +  \sum_{k=0}^{\Floor[(n-8)/6]} 
\mathrm{det}(M_{k+1}) y^{3k+3}U_{n-6k-7}(y) \\
& ~~~\qquad  + \sum_{k=0}^{\Floor[n-6/6]} \mathrm{det}(P_{k+1})(-y^{3k+2})
\left[U_{\tau,n-6k-4}(y) +yU_{\tau,n-6k-5}(y) \right].
\end{align*}
\end{theorem}

\begin {theorem} \label{thm:162534}
For any $n \geq 2$, let $\tau=\tau_1 \ldots \tau_{2a} \in S_{2a}$ where 
$\tau_1 \tau_3 \ldots \tau_{2a-1} = 123 \ldots a$ and $\tau_{2} \tau_4 \ldots 
\tau_{2a} = (2a) (2a-1) \ldots (a+1)$.
Then 
$$\mbox{NM}_\tau(t,x,y)=\left(\frac1{U_\tau(t,y)}\right)^x \text{ where }U_\tau(t,y)=1+\sum_{n\geq1}U_{\tau,n}(y)\frac{t^n}{n!},$$ with $U_{\tau,1}(y)=-y$, and for $n \geq 2,$
\begin{align*}
U_{\tau,n}(y)  = & ~  (1-y)U_{\tau,n-1}(y)  - 
\sum_{k=0}^{\lfloor (n-2a)/(2a)\rfloor} \binom{n-(k+1)a-1}{(k+1)a-1} y^{(k+1)a-1}
U_{\tau_a,n-(2(k+1)a)+1}(y) \\
& \qquad \quad +\sum_{k=0}^{\lfloor (n-2a-2)/(2a)\rfloor} \binom{n-(k+1)a-2}{(k+1)a} 
y^{(k+1)a} U_{\tau_a,n-(2(k+1)a)-1}(y).     
\end{align*}
\end{theorem}

We note that our results allows us to compute $\mbox{NM}_\tau(t,x,y)$ in two cases where  
$\tau = \tau_1 \ldots \tau_6$ and 
$\tau_1 =1$, $\tau_3 =2$, and $\tau_5 =3$.  Namely, the case where 
$\tau = 162534$ is consider in Theorem \ref{thm:142536} and the 
the case where $\tau = 142536$ is a special case of 
Theorem \ref{thm:162534}.  All such 
permutations have $\des(\tau) =2$. In fact, the first author in his thesis 
has computed 
$\mbox{NM}_\tau(t,x,y)$ in the other 4 cases where  $\tau = \tau_1 \ldots \tau_6$ and 
$\tau_1 =1$, $\tau_3 =2$, and $\tau_5 =3$ which we will not present 
here due to lack of space.

The outline of this paper is the following. In Section \ref{sec:sym}, 
we shall provide the necessary background on symmetric functions 
for our applications. In section \ref{sec:recip}, 
we shall recall the basic reciprocity method of 
\cite{JR,JR2,JR3} and \cite{BR} in the case 
where the permutations of $\Gamma$ are allowed to have more than 
one descent.  In Section \ref{sec:1-2-3}, we shall prove Theorem \ref{thm:1-2-3}. In Section \ref{sec:142536}, 
we shall prove Theorem \ref{thm:142536}. Finally, in Section 
\ref{sec:162534}, we shall prove Theorem \ref{thm:162534}.

\section{Symmetric Functions}\label{sec:sym}

In this section, we give the necessary background on 
symmetric functions that will be used in our proofs.

A partition of $n$ is a sequence of positive integers 
\begin{math}\la = (\la_1, \ldots ,\la_s)\end{math} such that 
\begin{math}0 < \la_1 \leq \cdots \leq \la_s\end{math} and $n=\la_1+ \cdots +\la_s$. We shall write $\lambda \vdash n$ to denote that $\lambda$ is 
partition of $n$ and we let $\ell(\lambda)$ denote the number of parts 
of $\lambda$. When a partition of $n$ involves repeated parts, 
we shall often use exponents in the partition notation to indicate 
these repeated parts. For example, we will write 
$(1^2,4^5)$ for the partition $(1,1,4,4,4,4,4)$.

Let \begin{math}\Lambda\end{math} denote the ring of symmetric functions in infinitely  
many variables \begin{math}x_1,x_2, \ldots \end{math}. The \begin{math}\nth\end{math} elementary symmetric function \begin{math}e_n = e_n(x_1,x_2, \ldots )\end{math}  and \begin{math}\nth\end{math} homogeneous 
 symmetric function \begin{math}h_n = h_n(x_1,x_2, \ldots )\end{math} are defined by the generating functions given in (\ref{genfns}). 
For any partition \begin{math}\la = (\la_1,\dots,\la_\ell)\end{math}, let \begin{math}e_\la = e_{\la_1} 
\cdots e_{\la_\ell}\end{math} and \begin{math}h_\la = h_{\la_1} 
\cdots h_{\la_\ell}\end{math}.  It is well known that \begin{math}e_0,e_1, \ldots \end{math} is 
an algebraically independent set of generators for \begin{math}\La\end{math}, and hence, 
a ring homomorphism \begin{math}\theta\end{math} on \begin{math}\Lambda\end{math} can be defined  
by simply specifying \begin{math}\theta(e_n)\end{math} for all \begin{math}n\end{math}.

If $\lambda =(\lambda_1, \ldots, \lambda_k)$ is a partition of $n$, 
then a $\lambda$-brick tabloid of shape $(n)$ is a filling 
of a rectangle consisting of $n$ cells with bricks of sizes 
$\lambda_1, \ldots, \lambda_k$ in such a way that no 
two bricks overlap. For example, Figure 
\ref{fig:DIMfig1} shows the six $(1^2,2^2)$-brick tabloids of 
shape $(6)$.

\fig{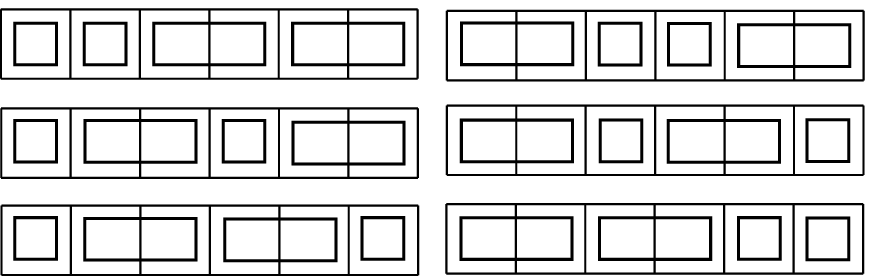}{The six $(1^2,2^2)$-brick tabloids of shape $(6)$.}

Let \begin{math}\mathcal{B}_{\la,n}\end{math} denote the set of \begin{math}\la\end{math}-brick tabloids 
of shape \begin{math}(n)\end{math} and let \begin{math}B_{\la,n}\end{math} be the number of \begin{math}\la\end{math}-brick 
tabloids of shape \begin{math}(n)\end{math}.  If \begin{math}B \in \mathcal{B}_{\la,n}\end{math}, we 
will write \begin{math}B =(b_1, \ldots, b_{\ell(\la)})\end{math} if the lengths of 
the bricks in \begin{math}B\end{math}, reading from left to right, are 
\begin{math}b_1, \ldots, b_{\ell(\la)}\end{math}. For example, the brick 
tabloid in the top right position in Figure \ref{fig:DIMfig1} is 
denoted as $(2,1,1,2)$. 
E\u{g}ecio\u{g}lu and the second author  \cite{Eg1} proved that 
\begin{equation}\label{htoe}
h_n = \sum_{\la \vdash n} (-1)^{n - \ell(\la)} B_{\la,n}~ e_\la.
\end{equation}
This interpretation of the expansion of 
$h_n$ in terms of the $e_{\lambda}$s will aid us in describing the coefficients of $\theta_\Gamma(H(t))=\mbox{U}_\Gamma(t,y)$ described 
in the next section, 
which will in turn allow us to compute the coefficients 
$\mbox{NM}_{\Gamma,n}(t,x,y)$.

\section{Extending the reciprocity method} \label{sec:recip}

Let $\Gamma$ be the set of permutations that all start with 1 and 
there is a $k \geq 1$ such that all $\sg \in \Gamma$ have 
$\des(\sg) \leq k$ and there is at least one $\tau \in \Gamma$ 
such that $\des(\tau) =k$. We want to give a combinatorial interpretation to \begin{equation}\label{eq:basic}
U_{\Gamma}(t,y)=\frac{1}{\mbox{NM}_{\Gamma}(t,1,y)}  = \frac{1}{1+ \sum_{n \geq 1} 
\frac{t^n}{n!} \mbox{NM}_{\Gamma,n}(1,y)}
\end{equation}
where 
$$\mbox{NM}_{\Gamma,n}(1,y) = \sum_{\sg \in \mathcal{NM}_n(\Gamma)} y^{1+\des(\sg)}.$$ 

We define a ring homomorphism $\theta_{\Gamma}$ on the ring of symmetric functions $\Lambda$ by setting $\theta_{\Gamma}(e_0) = 1$ and, for $n \geq 1,$ \begin{equation} \label{Theta} \theta_{\Gamma}(e_n) = \frac{(-1)^n}{n!} \mbox{NM}_{\Gamma,n}(1,y).
\end{equation}
It then follows that 
\begin{eqnarray} \label{theta=u}
\theta_{\Gamma}(H(t)) &=& \sum_{n \geq 0} \theta_{\Gamma}(h_n)t^n  = \frac{1}{\theta_{\tau}(E(-t))} = \frac{1}{1 + \sum_{n \geq 1} (-t)^n \theta_{\Gamma}(e_n)} \nonumber \\
&=& \frac{1}{1 + \sum_{n \geq 1} \frac{t^n}{n!} \mbox{NM}_{\Gamma,n}(1,y)} = 
U_{\Gamma}(t,y).
\end{eqnarray}

Thus $\displaystyle U_{\Gamma,n}(y) = n! \theta_\Gamma(h_n)$.  
Using (\ref{htoe}), we can compute \begin{eqnarray}\label{eq:basic1}
n! \theta_{\Gamma}(h_n) &=& n! \sum_{\la \vdash n} (-1)^{n-\ell(\la)} 
B_{\la,n}~ \theta_{\Gamma}(e_\la) \nonumber \\
&=& n! \sum_{\la \vdash n} (-1)^{n-\ell(\la)} \sum_{(b_1, \ldots, 
b_{\ell(\la)}) \in \mathcal{B}_{\la,n}} \prod_{i=1}^{\ell(\la)}  
\frac{(-1)^{b_i}}{b_i!} \mbox{NM}_{\Gamma,b_i}(1,y) \nonumber \\
&=& \sum_{\la \vdash n} (-1)^{\ell(\la)} \sum_{(b_1, \ldots, b_{\ell(\la)}) \in \mathcal{B}_{\la,n}} \binom{n}{b_1, \ldots, b_{\ell(\la)}}
\prod_{i=1}^{\ell(\la)}  \mbox{NM}_{\Gamma,b_i}(1,y).
\end{eqnarray}

To give combinatorial interpretation to the right hand side of (\ref{eq:basic1}), we select a brick tabloid $B = (b_1, b_2, \dots, b_{\ell(\la)} )$ of shape $(n)$ filled with bricks whose sizes induce the partition $\la$. We interpret the multinomial coefficient $\binom{n}{b_1, \ldots, b_{\ell(\la)}}$ as the number of ways to choose an ordered set partition $\mathcal{S} =(S_1, S_2, \dots, S_{\ell(\la)})$ of $\{1,2, \dots, n\}$ such that $|S_i| = b_i,$ for $i = 1, \dots, \ell(\la).$ For each brick $b_i,$ we then fill the cells of $b_i$ with numbers from $S_i$ such that the entries in the brick reduce to a permutation $\sg^{(i)} = \sg_1 \cdots \sg_{b_i}$ in $\mathcal{NM}_{b_i}(\Gamma).$ We label each descent of $\sg$ that occurs within each brick as well as the last cell of each brick by $y.$  This accounts for the factor $y^{\des(\sg^{(i)})+1}$ within each brick. Finally, we use the factor $(-1)^{\ell(\la)}$ to change the label of the last cell of each brick from $y$ to $-y$. We will denote the filled labeled brick tabloid constructed in this way as $\langle B,\mathcal{S},(\sg^{(1)}, \ldots, \sg^{(\ell(\la))})\rangle$. 

For example, when $n = 17, \Gamma = \{1324, 1423, 12345\},$ and 
$B = (9,3,5,2),$ consider the ordered set partition $\mathcal{S}=(S_1,S_2,S_3,S_4)$ of $\{1,2,\dots, 17\}$ where $S_1=\{2,5,6,9,11,15,16,17,19\},$ $S_2 = \{7,8,14\},$ $S_3 = \{1,3,10,13,18\},$ $S_4 = \{4,12\}$ and the permutations 
$\sg^{(1)} = 1~2~4~6~5~3~7~9~8 \in \mathcal{NM}_{9}(\Gamma),$ $\sg^{(2)} = 1~3~2 \in \mathcal{NM}_{7}(\Gamma),$ $\sg^{(3)} = 5~1~2~4~3 \in \mathcal{NM}_{5}(\Gamma),$ and $\sg^{(4)} = 2~1 \in \mathcal{NM}_{2}(\Gamma)$.  
Then the construction of 
$\langle B,\mathcal{S},(\sg^{(1)}, \ldots, \sg^{(4)})\rangle$ is 
pictured in Figure \ref{fig:CFLBTab}.

\begin{figure}[htbp]
  \begin{center}
    \includegraphics[width=0.6\textwidth]{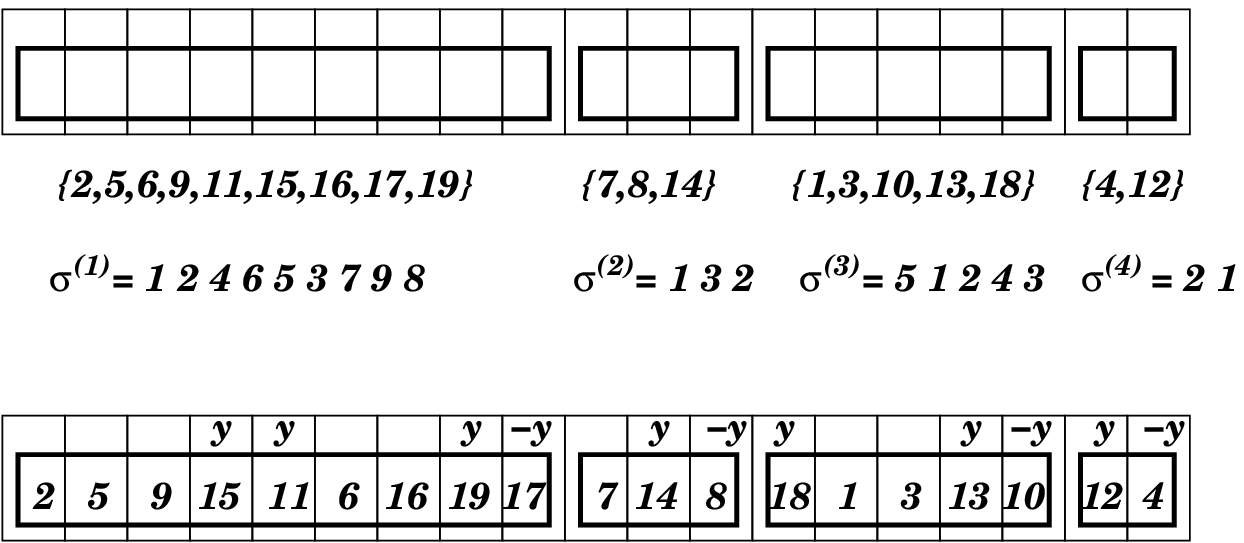}
    \caption{The construction of a filled-labeled-brick tabloid.}
    \label{fig:CFLBTab}
  \end{center}
\end{figure}


It is easy to see that we can recover 
the triple $ \langle B, (S_1, \dots, S_{\ell(\la)}), (\sg^{(1)}, \dots,\sg^{(\ell(\la))}  ) \rangle$ from $B$ and the permutation 
$\sg$ which is obtained by reading the entries in the 
cells from right to left.  We let $\mathcal{O}_{\Gamma, n}$ denote the set of all filled labeled brick tabloids created this way. That is, $\mathcal{O}_{\Gamma, n}$ consists of all pairs $O = (B, \sg )$ where 
\begin{enumerate}
\item $B = (b_1, b_2, \dots, b_{\ell(\la)})$ is a brick tabloid of shape $n$,  
\item $\sg = \sg_1  \cdots \sg_n$ is a permutation in $S_n$ such that there is no $\Gamma$-match of $\sg$ which 
lies entirely in a single brick of $B$, and 
\item if there is a cell $c$ such that a brick $b_i$ contains both cells $c$ and $c+1$ and $\sg_c > \sg_{c+1}$, then cell $c$ is labeled with a $y$ and the last cell of any brick is labeled with $-y$.
\end{enumerate}

We define the sign of each $O$ to be $sgn(O) = (-1)^{\ell(\la)}.$ The weight $W(O)$ of $O$ is defined to be the product of all the labels $y$ used in the brick. For example, 
the labeled brick tabloid pictured  Figure \ref{fig:CFLBTab} has 
$W(O) = y^{11}$ and $sgn(O) =(-1)^4 =1$. It follows that 
\begin{equation}\label{eq:basic2}
n!\theta_{\Gamma}(h_n) = \sum_{O \in \mathcal{O}_{\Gamma,n}} sgn(O) W(O).
\end{equation}

Next we define a sign-reversing, weight-preserving mapping $J_{\Gamma}: \mathcal{O}_{\Gamma, n} \rightarrow \mathcal{O}_{\Gamma, n}$ as 
follows.  Let $(B,\sg) \in \mathcal{O}_{\Gamma, n}$ where 
$B=(b_1, \ldots, b_k)$ and $\sg = \sg_1 \ldots \sg_n$. 
Then for any $i$, we let 
$\mbox{first}(b_i)$ be the element in the left-most cell of $b_i$ and 
$\mbox{last}(b_i)$ be the element in the right-most cell of $b_i$.
Then we read the cells of $(B,\sg)$ from left to right, looking for the first cell $c$ such that either \\
\ \\
{\bf Case I.}  cell $c$ is labeled with a $y$ in some brick $b_j$ and 
either {\bf (a)} $j=1$ or {\bf (b)} $j > 1$ and either 
{\bf (b.1)} $\mbox{last}(b_{j-1}) < \mbox{first}(b_j)$ or  {\bf (b.2)} 
 $\mbox{last}(b_{j-1}) > \mbox{first}(b_j)$ and there 
is $\tau$-match contained in the cells of $b_{j-1}$ and 
the cells $b_j$ that end weakly to the left of cell $c$ for 
some $\tau \in \Gamma$ or \\ 
\ \\
{\bf Case II.} cell $c$ is at the end of brick $b_i$ where $\sg_c > \sg_{c+1}$ and there is no $\Gamma$-match of $\sg$ that lies entirely in the cells of the bricks $b_i$ and $b_{i+1}$.\\

In Case I, we define $J_{\Gamma}((B,\sg))$ to be the filled labeled brick tabloid obtained from $(B,\sg)$ by breaking the brick $b_j$ that contains cell $c$ into two bricks $b_j'$ and $b_j''$ where $b_j'$ contains the cells of $b_j$ up to and including the cell $c$ while $b_j''$ contains the remaining cells of $b_j$. In addition, we change the label of cell $c$ from $y$ to $-y$. In Case II, 
$J_{\Gamma}((B,\sg))$ is obtained by combining the two 
bricks $b_i$ and $b_{i+1}$ 
into a single brick $b$ and changing the label of cell $c$ from $-y$ to $y$. If neither case occurs, then we let $J_{\Gamma}((B,\sg)) = (B,\sg)$.

\begin{figure}[h]
  \begin{center}
   \includegraphics[width=0.60\textwidth]{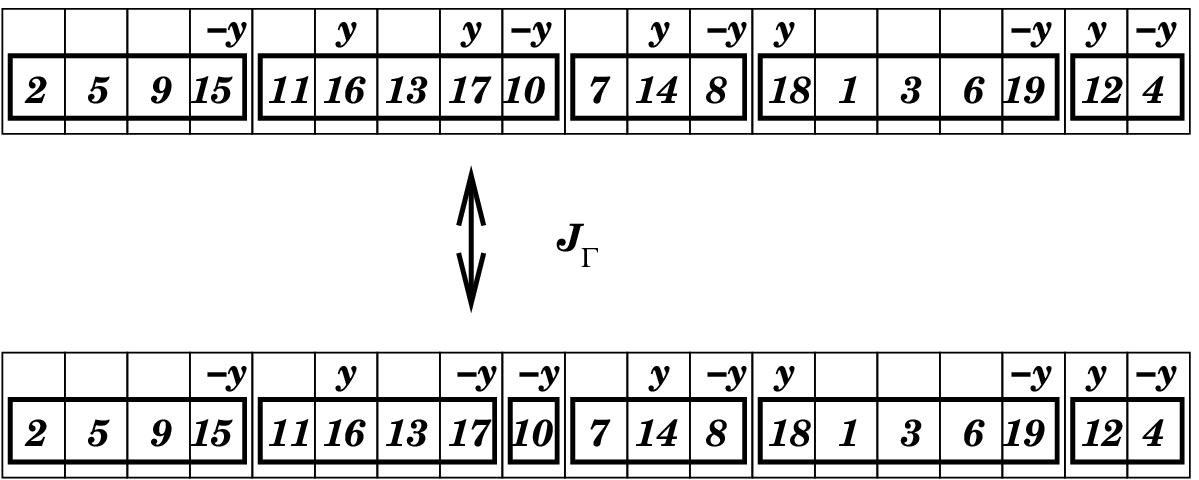}
   \caption{An example of the involution $J_{\Gamma}$.}
   \label{fig:JGamma}
  \end{center}
\end{figure}

For example, suppose $\Gamma = \{\tau\}$ where $\tau =14253$ and 
$(B,\sg) \in \mathcal{O}_{\Gamma,18}$ pictured at 
the top of Figure \ref{fig:JGamma}. 
We cannot use cell $c=4$ to define $J_{\Gamma}(B,\sg)$, 
because if we combined bricks $b_1$ and $b_2$, 
then $\red(9~15~11~16~13) = \tau$ would be a $\tau$-match 
contained in the resulting brick.  Similarly, 
we cannot use  cell $c=6$ to apply the involution because it fails to meet 
condition (b.2). In fact the first $c$ for which either Case I or 
Case II applies is cell $c=8$ so that $J_{\Gamma}(B,\sg)$ is 
equal to the $(B',\sg)$ pictured on the bottom of 
Figure \ref{fig:JGamma}.

We now prove that $J_{\Gamma}$ is an involution by showing $J_{\Gamma}^2$ is the identity mapping. Let $(B,\sg) \in \mathcal{O}_{\Gamma,n}$ where $B=(b_1, \ldots, b_k)$ and $\sg = \sg_1 \ldots \sg_n$. The key observation here is that applying the mapping $J_{\Gamma}$ to a brick in Case I will produce one in Case II, and vice versa. 

Suppose the filled, labeled brick tabloid $(B,\sg)$ is in Case I and its image $J_{\Gamma}((B,\sg))$ is obtained  by splitting some brick $b_j$ after cell $c$ into two bricks $b_j'$ and $b_j''.$  There are now two possibilities. \begin{itemize}

\item[(a)] $c$ is in the first brick $b_1$. In this case, $c$ must be the first cell which is labeled with $y$ so that the elements in $b_1'$ will be increasing. Furthermore, since we are assuming there is no $\Gamma$-match in the cells of brick $b_1$ in $(B,\sg)$, there cannot be any $\Gamma$-match that involves the cells of bricks $b_1'$ and $b_1''$ in $J_{\Gamma}((B,\sg))$. Hence, when we consider $J_{\Gamma}((B,\sg))$, the first possible cell where we can apply $J_{\Gamma}$ will be cell $c$ because we can now combine $b_1'$ and $b_1''$.  Thus, 
when we apply $J_{\Gamma}$ to $J_{\Gamma}((B,\sg))$, we will be 
in Case II using cell $c$ so that we will recombine bricks $b_1'$ and $b_1''$ into $b_1$ and replace the label of $-y$ on cell $c$ by $y$. Hence $J_{\Gamma}(J_{\Gamma}((B,\sg))) =(B,\sg)$ in this case. 

\item[(b)] $c$ is in brick $b_j$, where $j > 1$. Note that our definition of when a cell labeled $y$ can be used in Case I to define $J_{\Gamma}$ depends only on the cells and the brick structure to the left of that cell. Hence, we can not use any of the cells labeled $y$ to the left of $c$ to define $J_{\Gamma}(J_{\Gamma}((B,\sg)))$. Similarly, if we have two bricks $b_s$ and $b_{s+1}$ which lie entirely to the left of cell $c$ such that $\mbox{last}(b_s) = \sg_d > \mbox{first}(b_{s+1}) =\sg_{d+1}$, the criteria to use cell $d$ in the definition of $J_{\Gamma}$ on $J_{\Gamma}((B,\sg))$ depends only on the elements in bricks $b_s$ and $b_{s+1}$. Thus, the only cell $d$ which we could possibly use to define $J_{\Gamma}$ on $J_{\Gamma}((B,\sg))$ that lies to the left of $c$ is the last cell of $b_{j-1}$. However, our conditions that either $\mbox{last}(b_{j-1}) < \mbox{first}(b_{j}) = \mbox{first}(b_{j}')$ or  $\mbox{last}(b_{j-1}) > \mbox{first}(b_{j}) = \mbox{first}(b_{j}')$ with a $\Gamma$-match contained in the cells of $b_{j-1}$ and $b_j'$ force the first cell that can be used to define $J_{\Gamma}$ on $J_{\Gamma}((B,\sg))$ to be cell $c$. Thus, when we  apply $J_{\Gamma}$ to $J_{\Gamma}((B,\sg))$, we will be in Case II using cell $c$ and we will recombine bricks $b_j'$ and $b_j''$ into $b_j$ and replace the label of $-y$ on cell $c$ by $y$. Thus $J_{\Gamma}(J_{\Gamma}((B,\sg))) =(B,\sg)$ in this case. 
\end{itemize}

Suppose $(B,\sg)$ is in Case II and we define $J_{\Gamma}((B,\sg))$ at cell $c$, where $c$ is last cell of $b_j$ and $\sg_c > \sg_{c+1}$. Then by the same arguments that we used in Case I, there can be no cell labeled $y$ to the left of this cell $c$ in either $(B,\sg)$ or $J(B,\sg)$ which 
can be used to define the involution $J_\Gamma$. This follows from the fact that the brick structure before cell $c$ is unchanged between $(B,\sg)$ and $J(B,\sg)$. Similarly, there can be no two bricks that lie entirely to the left of cell $c$ in $J_{\Gamma}((B,\sg))$ that can be combined under $J_{\Gamma}$. Thus, the first cell that we can use to define $J_{\Gamma}$ to $J_{\Gamma}((B,\sg))$ is cell $c$ and it is easy to check that it satisfies the conditions of Case I.  Thus, when we apply $J_{\Gamma}$ to $J_{\Gamma}((B,\sg))$, 
we will be in Case I using cell $c$ and  we will combine bricks $b_j$ and $b_{j+1}$ into a single brick $b$ and replaced the label on cell $c$ by $y$. Then it is easy to see that when applying $J_{\Gamma}$ to $J_{\Gamma}((B,\sg))$, we will split $b$ back into bricks $b_j$ and $b_{j+1}$ and change the label on cell $c$ back to $-y$. Thus  $J_{\Gamma}(J_{\Gamma}((B,\sg))) =(B,\sg)$ in this case.

Hence $J_{\Gamma}$ is an involution. It is clear that if $J_{\Gamma}(B,\sg) \neq (B,\sg)$,  then $sgn(B,\sg)W(B,\sg) = -sgn(J_{\Gamma}(B,\sg))W(J_{\Gamma}(B,\sg)).$ Thus, it follows from (\ref{eq:basic2}) that 
\begin{equation}\label{eq:basic3} 
U_{\Gamma,n}(y) = n!\theta_\Gamma(h_n) =  \sum_{O \in \mathcal{O}_{\Gamma,n}} \sgn{O} W(O) = \sum_{O \in \mathcal{O}_{\Gamma,n}, J_{\Gamma}(O) =O} \sgn{O} W(O). 
\end{equation}
Thus, to compute $U_{\Gamma,n}(y)$, we must analyze the fixed points of $J_{\Gamma}$. Our next lemma characterizes the fixed points of $J_{\Gamma}$.

\begin{lemma} \label{lem:keyGamma} Let $B=(b_1, \ldots, b_k)$ be 
a brick tabloid of shape $(n)$ and $\sg = \sg_1 \ldots \sg_n \in S_n$. 
Then $(B,\sg)$ is a fixed point of $J_{\Gamma}$ if and only if it satisfies the following properties: 
\begin{description}
\item[(a)] if $i=1$ or $i > 1$ and $\mbox{last}(b_{i-1}) < \mbox{first}(b_i)$, then $b_i$ can have no cell labeled $y$ so that $\sg$ must be increasing in $b_i$, 

\item[(b)] if $i > 1$ and $\sg_e = \mbox{last}(b_{i-1}) > \mbox{first}(b_i)=\sg_{e+1}$, then there must be a $\Gamma$-match contained in the cells of $b_{i-1}$ and $b_i$ which must necessarily involve $\sg_e$ and $\sg_{e+1}$ and there can be at most $k-1$ cells labeled $y$ in $b_i$, and

\item[(c)] if $\Gamma$ has the property that, for all $\tau \in \Gamma$ such that $\des(\tau) = j \geq 1$, the bottom elements \footnote{If $\sg$ is a permutation with $\sg_i > \sg_{i+1}$, i.e. there is a descent in $\sg$ at position $i$, then we shall refer to $\sg_{i+1}$ as the bottom element of this descent.} of the descents in $\tau$ are $2, \ldots, j+1$,  when reading from left to right, 
then 
$$\mathrm{first}(b_1) < \mathrm{first}(b_2) < \cdots < \mathrm{first}(b_k).$$ 

\end{description}
\end{lemma}

\begin{proof} Suppose $(B,\sg)$ is a fixed point of $J_{\Gamma}$. Then it must be the case that in $(B,\sg)$, there is no cell $c$ to which either 
Case I or Case II applies. That is, when attempting to apply the involution $J_{\Gamma}$ to $(B,\sg)$, we cannot split any brick at a cell labeled $y$  and 
we cannot combine two consecutive bricks where the last cell of 
the first brick is larger than the first cell of the second brick.

For (a), note that if there is a cell labeled $y$ in $b_i$ and  $c$ is the left-most cell of $b_i$ labeled with $y$, then $c$ satisfies the conditions of Case I. Thus, there can be no cell labeled $y$ in $b_i$.

For (b), note that if there is no $\Gamma$-match contained in the cells of $b_{i-1}$ and $b_i$, then $e$ satisfies the conditions of Case II.  Thus, there must be a $\Gamma$-match contained in the cells of $b_{i-1}$ and $b_i$. If there are $k$ or more cells labeled $y$ in $b_i$, then let $c$ be the $k^{th}$ cell, reading from left to right, which is labeled with $y$. Then we know there is $\tau$-match contained in the cells of $b_{i-1}$ and $b_i$ which must necessarily involve $\sg_e$ and $\sg_{e+1}$ for some $\tau \in \Gamma$. But this $\tau$-match must end weakly before cell $c$ since otherwise $\tau$ would have at least $k+1$ descents. Thus $c$ would satisfy the conditions to apply Case I of our 
involution. Hence there can be no such $c$ which means that 
each such brick can contain at most $k-1$ descents.

To prove (c), suppose for a contradiction that 
there exist two consecutive bricks $b_i$ and $b_{i+1}$ 
such that $\sg_e = \mathrm{first}(b_{i}) > \mathrm{first}(b_{i+1}) =\sg_f$. There are two cases. \\
\ \\
{\bf Case A.}  $\sg$ is increasing in $b_i$. \\
Then $\sg_{f-1} = \mathrm{last}(b_i)$.  If $\sg_{f-1} < \sg_f$, then we know that $\sg_e \leq \sg_{f-1} < \sg_{f}$ which contradicts our choice of $\sg_e$ and $\sg_f$. Thus it must be the case that $\sg_{f-1} > \sg_f$.  But then there is $\tau \in \Gamma$ such that $\des(\tau) =j \geq 1$ and there is a $\tau$-match in the cells of $b_i$ and $b_{i+1}$ involving the $\sg_{f-1}$ and $\sg_f$. By our assumptions, $\sg_f$ can only play the role of $2$ in such a $\tau$-match. Hence there must be some $\sg_g$ with $e \leq g \leq f-2$ which plays the role of 1 in this $\tau$-match. But then we would have $\sg_e \leq \sg_g  < \sg_f$ which contradicts our choice of $\sg_e$ and $\sg_f$. Thus $\sg$ cannot be increasing in $b_i$. \\
\ \\
{\bf  Case B.} $\sg$ is not increasing in $b_i$. \\
In this case, by part (a), we know that it must be the case that $\sg_{e-1} = \mbox{last}(b_{i-1}) > \sg_e = \mbox{first}(b_i)$ and, by (b), there is $\tau \in \Gamma$ such that $\des(\tau) =j \geq 1$ and there is a $\tau$-match in the cells of $b_{i-1}$ and $b_{i}$ involving the cells $\sg_{e-1}$ and $\sg_e$. Call this $\tau$-match $\alpha$ and suppose that cell $h$ is the bottom element of the last descent in $\alpha$. It cannot be that $\sg_e =\sg_h$.  That is, there can be no cell labeled $y$ that occurs after cell $h$ in $b_i$ since otherwise the left-most such cell $c$ would satisfy the conditions of Case I of the definition of $J_{\Gamma}$. But this would mean that $\sg$ is increasing in $b_i$ starting at $\sg_h$ so that if $\sg_e =\sg_h$, then $\sg$ would be increasing in $b_i$ which contradicts our assumption in this case. Thus there is some $2 \leq i \leq j$ such that $\sg_{e}$ plays the role of $i$ in the $\tau$-match $\alpha$ and $\sg_h$ plays the role of $j+1$ in the $\tau$-match $\alpha$. But this means that $\sg_e$ is the smallest element in brick $b_i$. That is, let $\sg_c$ be the smallest element in  $b_i$. If $\sg_e \neq \sg_c$, then $\sg_c$ must be the bottom of some descent in $b_i$ which implies that $c \leq h$. But then $\sg_c$ is part of the $\tau$-match $\alpha$ which means that $\sg_c$ must be playing the role of one of $i+1, \ldots, j+1$ in the $\tau$-match $\alpha$ and $\sg_e$ is playing the role of $i$ in the $\tau$-match $\alpha$ which is impossible if $\sg_e \neq \sg_c$. It follows that $\sg_e \leq \sg_{f-1}$. Hence, it can not be that case that $\sg_{f-1} < \sg_f$ since otherwise $\sg_e < \sg_f$. Thus it must be the case that $\sg_{f-1} > \sg_f$. But this means that there exists some $\delta \in \Gamma$ such that $\des(\delta) =p \geq 1$ and there is a $\delta$-match in the cells of $b_i$ and $b_{i+1}$ involving the $\sg_{f-1}$ and $\sg_f$.  Call this $\delta$-match $\beta$. By assumption, the bottom elements of the descents in $\delta$ are $2,3, \ldots, p+1$ so that $\sg_f$ must be playing the role of $2,3, \ldots ,p+1$ in the $\delta$-match $\beta$. Let $\sg_g$ be the element that plays the role of $1$ in the $\delta$-match $\beta$. $\sg_g$ must be in $b_i$ since $\delta$ must start with 1. But then we would have that $\sg_e \leq \sg_g < \sg_f$ since $\sg_e$ is the smallest element in $b_i$. \\

Thus, both Case A and Case B are impossible. 
Hence we must have that 
$$\mathrm{first}(b_1) < \mathrm{first}(b_2) < \cdots < \mathrm{first}(b_k).$$  
\end{proof}

We note that if condition (3) of the Lemma fails, 
it may be that the first elements 
of the bricks do not form an increasing sequence.
For  example, it is easy to check that if $\Gamma = \{15342\}$, 
then the $(B,\sg)$ pictured in Figure \ref{fig:Counter2} is 
such a fixed point of $J_{\Gamma}$.

\begin{figure}[h]
  \begin{center}
   \includegraphics[width=0.60\textwidth]{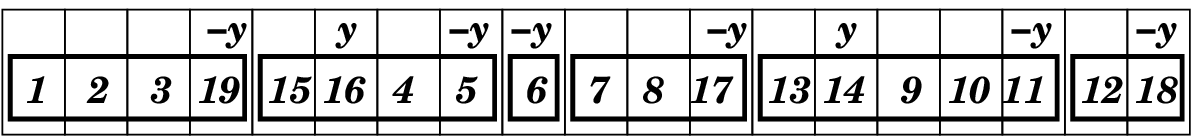}
   \caption{A fixed point of $J_{\{15342\}}$.}
   \label{fig:Counter2}
  \end{center}
\end{figure}

\section{The proof of Theorem \ref{thm:1-2-3}  }\label{sec:1-2-3}

In this section, we shall prove Theorem \ref{thm:1-2-3} which is the simplest case 
of our three examples. For convenience, 
we first restate the statement of Theorem \ref{thm:1-2-3}.

\begin {thm} 
Let $\Gamma = \{14253,15243\}$. Then 
$$\mbox{NM}_\Gamma(t,x,y)=\left(\frac1{U_\Gamma(t,y)}\right)^x \text{ where }U_\Gamma(t,y)=1+\sum_{n\geq1}U_{\Gamma,n}(y)\frac{t^n}{n!},$$ with $U_{\Gamma,1}(y)=-y$, and for $n \geq 2,$
\begin{align*}
\displaystyle  U_{\Gamma,n}(y) & = (1-y)U_{\Gamma,n-1}(y)  -y^2(n-3)\left(U_{\Gamma, n-4}(y) +(1-y)(n-5)U_{\Gamma, n-5}(y)  \right)\\
& \qquad \qquad \qquad \qquad \qquad -y^3(n-3)(n-5)(n-6)U_{\Gamma,n-6}(y).
\end{align*}

\end{thm}

\emph{Proof.} Let $\Gamma = \{14253,15243\},$ we need to show that the polynomials $$U_{\Gamma,n }(y) = \sum_{O \in \mathcal{O}_{\Gamma,n}, J_{\Gamma}(O) =O} \sgn{O} W(O)$$ satisfy the following properties: \begin{enumerate}
\item $U_{\tau,1}(y)=-y$, and 
\item for $n \geq 2,$ \begin{align*} \label{equ:1-2-3}
\displaystyle  U_{\Gamma,n}(y) & = (1-y)U_{\Gamma,n-1}(y)  -y^2(n-3)\left(U_{\Gamma, n-4}(y) +(1-y)(n-5)U_{\Gamma, n-5}(y)  \right)\\
& \qquad \qquad \qquad \qquad \qquad -y^3(n-3)(n-5)(n-6)U_{\Gamma,n-6}(y).
\end{align*} 
\end{enumerate}  

It is easy to see when $n=1,$ the only fixed point comes from 
brick tabloid that has a single brick of size 1 which contains 1 and 
the label on cell 1 is $-y.$ Thus 
$U_{\tau,1}(y)=-y.$

For $n \geq 2,$ let $O=(B,\sg)$ be a fixed point of $J_{\Gamma}$ where 
$B=(b_1, \ldots, b_k)$ and $\sg=\sg_1 \cdots \sg_n$. 
First we show that 1 must be in the first cell of $B$.  That is, if 
$1 = \sg_c$ where $c > 1$, then $\sg_{c-1} > \sg_c$. We claim 
that whenever we have a descent $\sg_i > \sg_{i+1}$ in $\sg$, then 
$\sg_i$ and $\sg_{i+1}$ must be part of a $\Gamma$-match in $\sg$. That is, 
it is either the case that (i) there are bricks $b_s$ and $b_{s+1}$ such that 
$\sg_i$ is the last cell of $b_s$ and $\sg_{i+1}$ 
is the first cell of $b_{s+1}$ or (ii) there is a 
brick $b_s$ that contains both $\sg_i$ and $\sg_{i+1}$. In case (i), condition 
3 of Lemma \ref{lem:keyGamma} ensures that $\sg_i$ and $\sg_{i+1}$ must 
be part of $\Gamma$-match. In case (ii), we know that cell $i$ is 
labeled with $y$.  It follows from condition (2) of Lemma \ref{lem:keyGamma} 
that it can not be that either $s =1$ so that $b_s =b_1$ or that 
$s > 1$ and $\mathrm{last}(b_{s-1}) < \mathrm{first}(b_s)$ because  
those conditions force  that $\sg$ is increasing in $b_s$.  
Thus we must have that $s > 1$ and 
$\mathrm{last}(b_{s-1}) > \mathrm{first}(b_s)$. 
Since $(B,\sg)$ is a fixed point of $J_{\Gamma}$, it cannot 
be that there is a $\Gamma$-match in $\sg$ 
which includes $\mathrm{last}(b_{s-1})$ 
and  $\mathrm{first}(b_s)$ that ends weakly to the left of $\sg_i$ because 
then cell $i$ would satisfy Case I of our definition of $J_\Gamma$ and, 
hence, $(B,\sg)$ would not be a fixed point of $J_\Gamma$. 
Thus the $\Gamma$-match which includes $\mathrm{last}(b_{s-1})$ 
and  $\mathrm{first}(b_s)$ must involve $\sg_i$ and $\sg_{i+1}$. 
However, there can be no $\Gamma$-match that involves $\sg_{c-1}$ and 
$\sg_c$ since $\sg_c=1$ can only play the role of 1 in a $\Gamma$-match 
and each element of $\Gamma$ starts with 1. Thus we must have $\sg_1 =1$.

Next we claim that 2 must be in either cell 2 or cell 3 in $O.$ 
For a contradiction, assume that 2 is in cell $c$ for $c > 3$. 
Then once again $\sg_{c-1} > \sg_c$ so that there must be a $\Gamma$-match 
in $\sg$ that involves the two cells $c-1$ and $c$ in $(B,\sg)$. 
However, In this case, the number which is in cell $c-2$ must be greater than 
$\sg_c$ so that the only possible $\Gamma$-match that involves 2 must start from cell $c$ where 2 plays the role of 1 in the match. Thus there is no $\Gamma$-match in 
$\sg$ that involves $\sg_{c-1}$ and $\sg_c$.

We now have 
have two cases.\\
\ \\
{\bf Case 1.} 2 is in cell 2 of $O$.\\
\ \\
In this case there are two possibilities, namely, either (i) 1 and 2 are both in the first brick $b_1$ of $(B,\sg)$ or (ii) brick $b_1$ is a single cell filled with 1, and 2 is in the first cell of the second brick $b_2$ of $O$.  In either case, we know that 1 is not part of a 
$\Gamma$-match in $\sg$. So if we remove cell 1 from $O$ and subtract $1$ from the elements in the remaining cells, we will obtain a fixed point $O'$ of $J_{\Gamma}$ in $\mathcal{O}_{\Gamma,n-1}.$ 

Moreover, we can create a fixed point $O=(B,\sg) \in \mathcal{O}_n$ of $J_{\Gamma}$ satisfying the three conditions of Lemma \ref{lem:keyGamma} where $\sg_2 =2$ by starting with a fixed point $(B',\sg') \in \mathcal{O}_{\Gamma,n-1}$ of $J_\Gamma$, where $B' =(b_1', \ldots, b_r')$ and $\sg' =\sg_1' \cdots \sg_{n-1}'$, and then letting $\sg = 1 (\sg_1'+1) \cdots (\sg_{n-1}' +1)$, and setting  $B = (1,b_1', \ldots, b_r')$ or setting $B = (1+b_1', \ldots, b_r')$.

It follows that fixed points in Case 1 will contribute $(1-y)U_{\Gamma,n-1}(y)$ to  $U_{\Gamma,n}(y)$.

\ \\
{\bf Case 2.} 2 is in cell 3 of $O=(B,\sg)$. \\
\ \\
Since there is no decrease within the first brick $b_1$ of $O=(B,\sg),$ it must be the case that 2 is in the first cell of brick $b_2$ and there must be either a $14253$-match or a $15243$-match that involves the cells of the first two bricks. Therefore, we know that brick $b_2$ has at least 3 cells. In addition, we claim that 3 is in cell 5 of $O$ since otherwise, 3 must be in some cell $c$ for $c > 6$ and there must be a $\Gamma$-match between the two cells $c-1$ and $c$ in $O.$ By the previous argument, we can see that if 3 is too far away from 1 and 2, then it must play the role of 1 in any match that involves cell $c.$ Thus, the only possible $\Gamma$-match that contains cell $c$ must also start at $c$ and can never involve both cells $c-1$ and $c.$ Also, 3 cannot be in cell 2 nor 4 in $O$ since both $\sg_2$ and $\sg_4$ are greater than 3, due to the $\Gamma$-match starting from cell 1. We now have two subcases depending on whether or not there is a $\Gamma$-match in $O$ starting at cell 3.  

\ \\
{\bf Subcase 2.a.} There is no $\Gamma$-match in $O$ starting at cell 3. \\
\ \\  
In this case, we first choose a number $x$ to fill in cell 2 of $O.$ There are $n-3$ choices
for $x$.
 For each choice of $\sg_2 =x$, we let $d$ be the smallest of the remaining numbers, that is, 
$$d = \min\left( \{1,2, \ldots,n\} - \{1,2,3,\sg_2\} \right).$$ 
We claim that $d$ must be either in cell 4 or cell 6 in $(B,\sg)$. First, $d$ cannot 
be in cell 7 since otherwise there would be a $\Gamma$-match in $\sg$ starting 
at cell 3.  Next $d$ cannot be a cell $c$ where $c >7$ since otherwise 
$\sg_{c-1} >\sg_c =d$ which means that there must be a $\Gamma$-match in $\sg$ which 
includes both $\sg_{c-1}$ and $\sg_c$.  However, in the case, we would 
also have $\sg_{c-2}> \sg_c$ which implies the only role that $\sg_c$ can play 
in a $\Gamma$-match is 1. 

This leaves us with three possibilities which are pictured in Figure 
\ref{fig:3poss}. That is, either (i) $d$ is in cell 4, (ii) $d$ is in 
cell $6$ and is in brick $b_2$ or (iii) $d$ is in cell $6$, but is the 
first element of brick $b_3$. In case (i), we can remove that first four 
cells from $B$, reduce the remaining elements of $\sg$ to obtain a 
permutation $\alpha \in S_{n-4}$, and let $B'=(b_2-2,b_3, \ldots,b_k)$ to 
obtain a fixed point $(B',\alpha)$ of $J_{\Gamma}$ of size $n-4$. 
Such fixed points will contribute $-y^2U_{\Gamma,n-4}(y)$ to $U_{\Gamma,n}(y)$. 
In case (ii), we have $(n-5)$ ways to choose the element $z$ in 
cell 4. Then we can remove that first five cells  
cells from $B$, reduce the remaining elements of $\sg$ to obtain a 
permutation $\alpha \in S_{n-5}$, and let $B'=(b_2-3,b_3, \ldots,b_k)$ to 
obtain a fixed point $(B',\alpha)$ of $J_{\Gamma}$ of size $n-5$. 
Such fixed points will contribute $-y^2U_{\Gamma,n-5}(y)$ to $U_{\Gamma,n}(y)$. 
In case (iii), we have $(n-5)$ ways to choose the element $z$ in 
cell 4. Then we can remove that first five cells  
cells from $B$, reduce the remaining elements of $\sg$ to obtain a 
permutation $\alpha \in S_{n-5}$, and let $B'=(b_2-3,b_3, \ldots,b_k)$ to 
obtain a fixed point $(B',\alpha)$ of $J_{\Gamma}$ of size $n-5$. 
Such fixed points will contribute $y^3U_{\Gamma,n-5}(y)$ to $U_{\Gamma,n}(y)$. 
Therefore, the total contribution of the fixed points from Subcase 2.a. is $$ -y^2(n-3)\left( U_{\Gamma, n-4}(y) +(1-y)(n-5)U_{\Gamma, n-5}(y)  \right). $$

\begin{figure}[h]
  \begin{center}
   \includegraphics[width=0.25\textwidth]{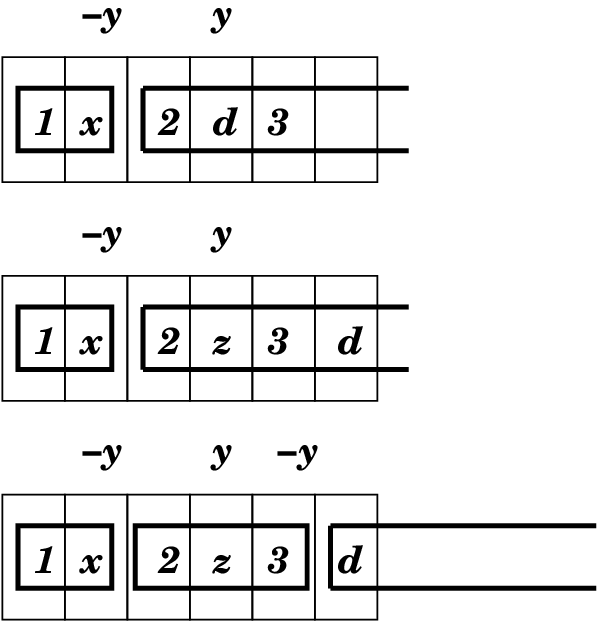}
   \caption{The possible choice for $d$ in Subcase 2a.}
   \label{fig:3poss}
  \end{center}
\end{figure}

\noindent 
{\bf Subcase 2.b.} There is a $\Gamma$-match in $O$ starting at cell 3. \\
\ \\ 
In this case, we first choose a number $x$ to fill in cell 2 of $O.$ There are $n-3$ choices
for $x$. For each choice of $\sg_2,$ let 
$$d = \min\left( \{1, \ldots,n\}- \{1,2,3,\sg_2\} \right).$$
Then we claim that $d$ must be in cell $7$.  That is, we can argue as in 
Subcase 2a that it cannot be that $d$ in cell $c$ for $c >7$.  But 
since there is a $\Gamma$-match starting at cell $3$ we know 
$\sg_4 > \sg_7$ and $\sg_6 > \sg_7$ so that $d$ cannot be in cells 4 or 6. 
We then have $(n-5)(n-6)$ ways to choose $\sg_4 =z$ and $\sg_6 =a$. 

Next by condition (b) of Lemma \ref{lem:keyGamma}, we 
know that each brick in $b$ in $B$ can contain at most one descent. Since 
we know that $b_2$ must have size at least 3 because there is a $\Gamma$-match in $\sg$
starting at cell 1 which is contained in $b_1$ and $b_2$, this means that either 
$b_2 =3$ or $b_2=4$. We claim that $b_2$ is of size $4$.  That is, 
if $b_2 =3$, then either (I) $a > d$ are in $b_3$ or (II) brick $b_3$ contains 
a single cell containing $a$ and $d$ is the first cell of $b_4$.  Case (I) 
cannot happen because then $\mathrm{last}(b_2) =3 < \mathrm{first}(b_3) =a$ which 
implies that the elements in $b_3$ must be increasing by  
condition (a) of Lemma \ref{lem:keyGamma}.  Case (II) cannot happen 
because that  $\mathrm{last}(b_3) =a > \mathrm{first}(b_4) =d$ which implies 
there must be a $\Gamma$-match contained in the cells of $b_3$ and $b_4$ 
which involves both $\sg_6 =a$ and $\sg_7 =d$ which is impossible since 
$a > d$.  Thus we are in the situation pictured in Figure \ref{fig:Subcase2b}.

\begin{figure}[h]
  \begin{center}
   \includegraphics[width=0.25\textwidth]{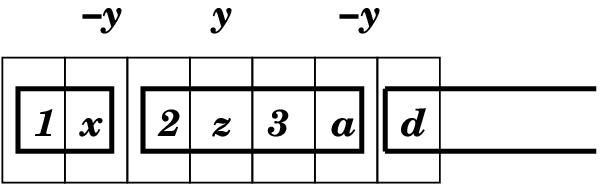}
   \caption{Subcase 2b.}
   \label{fig:Subcase2b}
  \end{center}
\end{figure}

Then we can remove that first six cells  
cells from $B$, reduce the remaining elements of $\sg$ to obtain a 
permutation $\alpha \in S_{n-6}$, and let $B'=(b_3, \ldots,b_k)$ to 
obtain a fixed point $(B',\alpha)$ of $J_{\Gamma}$ of size $n-6$. 
Such fixed points will contribute $(n-3)(n-5)(n-6)y^3U_{\Gamma,n-6}(y)$ to $U_{\Gamma,n}(y)$.

In total, we obtain the recursion for $U_{\Gamma,n}(y)$ as follows. \begin{align*} 
\displaystyle  U_{\Gamma,n}(y) & = (1-y)U_{\Gamma,n-1}(y)  -y^2(n-3)\left(U_{\Gamma, n-4}(y) +(1-y)(n-5)U_{\Gamma, n-5}(y)  \right)+ \\
& \qquad \qquad \qquad \qquad \qquad y^3(n-3)(n-5)(n-6)U_{\Gamma,n-6}(y).
\end{align*} 
This proves Theorem \ref{thm:1-2-3}. \qed

Using Theorem \ref{thm:1-2-3}, we computed the initial values of 
the $U_{\Gamma,n}(y)$s which are given in Table 1.

\begin{table}[h]
\begin{center}
\begin{tabular}{c|l}
n & $U_{\Gamma,n}(y)$ \\ 
\hline 1 & $-y$\\
2 & $-y + y^2$ \\
3 & $-y + 2y^2 - y^3$\\
4 & $-y + 3 y^2 - 3 y^3 + y^4$\\
5 & $-y + 4 y^2 -4y^3 + 4 y^4 -y^5$\\
6 & $-y + 5y^2 -2y^3 + 2y^4 -5y^5 + y^6$ \\ 
7 & $-y + 6 y^2 + 5y^3 -28 y^4 + 5 y^5 + 6 y^6 - y^7$ \\
8 & $-y + 7 y^2 + 19 y^3 - 123y^4 + 123y^5 -19 y^6 -7 y^7 + y^8$ \\ 
9 & $-y + 8 y^2 + 42y^3 -334y^4 + 588 y^5 -334y^6 + 42y^7 + 8 y^8 - y^9$  \\ 
10 & $-y + 9 y^2 + 76y^3 -726y^4 + 1606y^5 -1606y^6 + 726 y^7 - 76 y^8 -
 9 y^9 + y^{10}$  
\end{tabular} 
\end{center}
\caption{The polynomials $U_{\Gamma,n}(-y)$ for $\Gamma=\{14253,15243\}$} \label{tab:Gamma}
\end{table}

Using these initial values of 
the $U_{\Gamma,n}(y)$s, one can then compute the initial values of 
$NM_{\Gamma,n}(x,y)$ which are given in Table 2.

\begin{table}[h]
\begin{center}
\begin{tabular}{c|l}
n & $NM_{\Gamma,n}(x,y)$ \\ 
\hline 1 & $xy$\\
2 & $xy+x^2y^2$ \\
3 & $x y+x y^2+3 x^2 y^2+x^3 y^3$\\
4 & $x y+4 x y^2+7 x^2 y^2+x y^3+4 x^2 y^3+6 x^3 y^3+x^4 y^4$\\
5 & $x y+11 x y^2+15 x^2 y^2+9 x y^3+30 x^2 y^3+25 x^3 y^3+x y^4+5 x^2 y^4+10 x^3 y^4+10 x^4 y^4+x^5 y^5$\\
6 & $x y+26 x y^2+31 x^2 y^2+58 x y^3+146 x^2 y^3+90 x^3 y^3+22 x y^4+79 x^2 y^4+120 x^3 y^4+ $ \\ 
& $65 x^4 y^4+x y^5+6 x^2 y^5+15 x^3 y^5+20 x^4 y^5+15 x^5
y^5+x^6 y^6$ \\
7 & $x y+57 x y^2+63 x^2 y^2+282 x y^3+588 x^2 y^3+301 x^3 y^3+252 x y^4+770 x^2 y^4+ $ \\
& $896 x^3 y^4+350 x^4 y^4+51 x y^5+210 x^2 y^5+364 x^3 y^5+350 x^4
y^5+140 x^5 y^5+$\\
& $x y^6+7 x^2 y^6+21 x^3 y^6+35 x^4 y^6+35 x^5 y^6+21 x^6 y^6+x^7 y^7$
\end{tabular} 
\end{center}
\caption{The polynomials $MN_{\Gamma,n}(x,y)$ for $\Gamma=\{14253,15243\}$} \label{tab:NMGamma}
\end{table}

\section{The generating function $U_{142536}(t,y)$.} \label{sec:142536}

 In this section, we shall study the generating function 
$U_\tau(t,y)$ where $\tau = 142536$. We let $J_\tau$ denote the 
involution $J_\Gamma$ from Section 3 where $\Gamma =\{\tau\}$.

We claim that the polynomials 
$$U_{\tau,n }(y) = \sum_{O \in \mathcal{O}_{\tau,n}, J_{\tau}(O) =O} \sgn{O} W(O)$$ 
satisfy the following properties: \begin{enumerate}
\item $U_{\tau,1}(y)=-y$, and 
\item for $n \geq 2,$  
\begin{align*}
\displaystyle U_{\tau,n}(y) = &~~~  (1-y)U_{\Gamma,n-1}(y) +  \sum_{k=0}^{\Floor[(n-8)/6]} 
\mathrm{det}(M_{k+1}) y^{3k+3}U_{n-6k-7}(y) \\
& ~~~\qquad  + \sum_{k=0}^{\Floor[(n-6)/6]} \mathrm{det}(P_{k+1})(-y^{3k+2})
\left[U_{\tau,n-6k-4}(y) +yU_{\tau,n-6k-5}(y) \right].
\end{align*}
\end{enumerate}  

It is easy to see when $n=1,$ the only fixed point comes from 
brick tabloid that has a single brick of size 1 which contains 1 and 
the label on cell 1 is $-y.$ Thus 
$U_{\tau,1}(y)=-y.$

For $n \geq 2,$ let $O=(B,\sg)$ be a fixed point of $I_{\Gamma}$ where $B=(b_1, \ldots, b_k)$ and $\sg=\sg_1 \cdots \sg_n$. 
First we show that 1 must be in the first cell of $B$.  That is, if 
$1 = \sg_c$ where $c > 1$, then $\sg_{c-1} > \sg_c$. We claim 
that whenever we have a descent $\sg_i > \sg_{i+1}$ in $\sg$, then 
$\sg_i$ and $\sg_{i+1}$ must be part of a $\tau$-match in $\sg$. That is, 
it is either the case that (i) there are bricks $b_s$ and $b_{s+1}$ such that 
$\sg_i$ is the last cell of $b_s$ and $\sg_{i+1}$ 
is the first cell of $b_{s+1}$ or (ii) there is a 
brick $b_s$ that contains both $\sg_i$ and $\sg_{i+1}$. In case (i), condition 
3 of Lemma \ref{lem:keyGamma} ensures that $\sg_i$ and $\sg_{i+1}$ must 
be part of $\tau$-match. In case (ii), we know that cell $i$ is 
labeled with $y$.  It follows from condition (2) of Lemma \ref{lem:keyGamma} 
that it can not be that either $s =1$ so that $b_s =b_1$ or that 
$s > 1$ and $\mathrm{last}(b_{s-1}) < \mathrm{first}(b_s)$ because  
those conditions force  that $\sg$ is increasing in $b_s$.  
Thus we must have that $s > 1$ and 
$\mathrm{last}(b_{s-1}) > \mathrm{first}(b_s)$. 
Since $(B,\sg)$ is a fixed point of $J_{\tau}$, it cannot 
be that there is a $\tau$-match in $\sg$ 
which includes $\mathrm{last}(b_{s-1})$ 
and  $\mathrm{first}(b_s)$ that ends weakly to the left of $\sg_i$ because 
then cell $i$ would satisfy Case I of our definition of $J_\tau$ and, 
hence, $(B,\sg)$ would not be a fixed point of $J_\tau$. 
Thus the $\tau$-match which includes $\mathrm{last}(b_{s-1})$ 
and  $\mathrm{first}(b_s)$ must involve $\sg_i$ and $\sg_{i+1}$. 
However, there can be no $\tau$-match that involves $\sg_{c-1}$ and 
$\sg_c$ since $\sg_c=1$ can only play the role of 1 in $\tau$-match 
and $\tau$ starts with 1. Thus we must have $\sg_1 =1$.

Next we claim that 2 must be in either cell 2 or cell 3 in $O.$ 
For a contradiction, assume that 2 is in cell $c$ for $c > 3$. 
Then once again $\sg_{c-1} > \sg_c$ so that there must be a $\tau$-match 
in $\sg$ that involves the two cells $c-1$ and $c$ in $(B,\sg).$ However, since 2 is too far from 1 in $B,$ the only possible 142536-match that involves 2 must start from cell $c$ where 2 plays the role of 1 in the match. We then 
have two cases.\\
\ \\
{\bf Case 1.} 2 is in cell 2 of $O$.\\
\ \\
In this case, there are two possibilities, namely, either (i) 1 and 2 are both in the first brick $b_1$ of $(B,\sg)$ or (ii) brick $b_1$ is a single cell filled with 1 and 2 is in the first cell of the second brick $b_2$ of $(B,\sg)$.  In either case, we know that 1 is not part of a $\tau$-match in $(B,\sg)$. So if we remove cell 1 from $(B,\sg)$ and subtract $1$ from the elements in the remaining cells, we will obtain a fixed point $(B',\sg')$ of $J_{\Gamma}$ in $\mathcal{O}_{\Gamma,n-1}.$ 

Moreover, we can create a fixed point $O=(B,\sg) \in \mathcal{O}_n$ satisfying the three conditions of Lemma \ref{lem:keyGamma} where $\sg_2 =2$ by starting with a fixed point $(B',\sg') \in \mathcal{O}_{\Gamma,n-1}$ of $J_\Gamma$, where $B' =(b_1', \ldots, b_r')$ and $\sg' =\sg_1' \cdots \sg_{n-1}'$, and then letting $\sg = 1 (\sg_1'+1) \cdots (\sg_{n-1}' +1)$, and setting  $B = (1,b_1', \ldots, b_r')$ or setting $B = (1+b_1', \ldots, b_r')$.\\

It follows that fixed points in Case 1 will contribute 
$(1-y)U_{\Gamma,n-1}(y)$ to $U_{\Gamma,n}(y)$. \\
\ \\
{\bf Case 2.} 2 is in cell 3 of $O=(B,\sg)$. \\
\ \\
Since there is no decrease within the first brick $b_1$ of $O=(B,\sg),$ it must be the case that 2 is in the first cell of brick $b_2$ and there must be a 142536-match that involves the cells of the first two bricks. Therefore, we know that brick $b_2$ has at least 4 cells.

To analyze this case, it will be useful to picture $O=(B,\sg)$ as a 2-line array $A(O)$ where the elements in the $i$-th column are $\sg_{2i-1}$ 
and $\sg_{2i}$ reading from bottom to top.  In $A(O)$, imagine 
the we draw an directed arrow from the cell containing $i$ to the cell 
containing $i+1$.  Then it is easy to see that a $\tau$-match 
correspond to block of points as pictured in Figure \ref{fig:tauarray}

\begin{figure}[h]
  \begin{center}
   \includegraphics[width=0.10\textwidth]{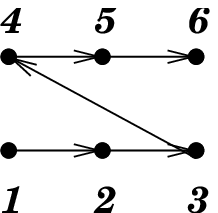}
   \caption{A $142536$-match as a 2-line array.}
   \label{fig:tauarray}
  \end{center}
\end{figure}

Now imagine that $A(0)$ starts with series of $\tau$-matches starting 
at positions $1,3,5, \ldots$. We have pictured this situation at 
the top of Figure \ref{fig:142536}. Now consider the brick 
structure of $O=(B,\sg)$. Since the elements of $b_1$ must be 
increasing and $\sg_2 > \sg_3$, it must be the case that $b_1 =2$ and 
$b_2 \geq 4$. We claim that $b_2 =4$ because if $b_2 > 4$, 
then $\sg_6 < \sg_7$ would be a descent in $b_2$. 
Thus cell 6 would be labeled with a $y$. The $\tau$-match 
starting at cell 1 ends a cell $6$ so that cell 6 would satisfy  Case I of our definition of $J_\tau$ which 
contracts that the fact that  $O=(B,\sg)$ is a fixed point of $J_\tau$. 
Now the fact that $\sg_6 > \sg_7$ implies that $b_3 \geq 2$ since 
there must be a $\tau$-match that involves $\sg_6$ and $\sg_7$. Now if 
there is a $\tau$-match starting at cell 7, then we can see 
that $\sg_8 > \sg_9$.  It cannot be that $\sg_8$ and $\sg_9$ are 
both in $b_3$ because it would follow that cell 8 would be labeled 
with a $y$ and the $\tau$-match starting at $\sg_3$ would 
end at cell 8. Thus cell 8  would be in Case I of our definition of $J_\tau$ 
which contracts that the fact that  $O=(B,\sg)$ is a fixed point of $J_\tau$. 
Thus it must be the case that $b_3 =2$. But the $\tau$-match 
starting at cell 7 forces $\sg_8 > \sg_9$ so that there 
is a decrease between $\mathrm{last}(b_3)$ and $\mathrm{first}(b_4)$ which 
implies that there is $\tau$ contained in $b_3$ and $b_4$, which then means 
that $b_4 \geq 4$. Now if there is a $\tau$-matches starting 
at $\sg_9$, then it must be the case that $\sg_{12} > \sg_{13}$. Hence, 
it cannot be $b_4 > 4$ since otherwise cell 12 is labeled with a $y$. 
Since the $\tau$-match starting a cell 7 ends at cell 12, then cell 12  
would be in Case I of our definition of $J_\tau$ 
which contracts that the fact that  $O=(B,\sg)$ is a fixed point of $J_\tau$. 
Thus it must be the case that $b_4 =4$.  We can continue 
to reason in this way to conclude that if 
there are $\tau$-matches starting at cells $1,3,7,9, \ldots, 6k+1,6k+3$, 
then $b_{2i-1} =2$ for $i =1, ,2k+1$ and $b_{2i}=4$ for $i =1, 
\ldots, 2k$. Similarly, if there  are $\tau$-matches starting at cells 
$1,3,7,9, \ldots, 6k+1$ but no $\tau$-match starting at cell $6k+3$, 
then $b_{2i-1} =2$ for $i =1, ,2k$ and $b_{2i}=4$ for $i =1, 
\ldots, 2k-1$ and $b_{2k} \geq 4$. 

\begin{figure}[h]
  \begin{center}
   \includegraphics[width=0.30\textwidth]{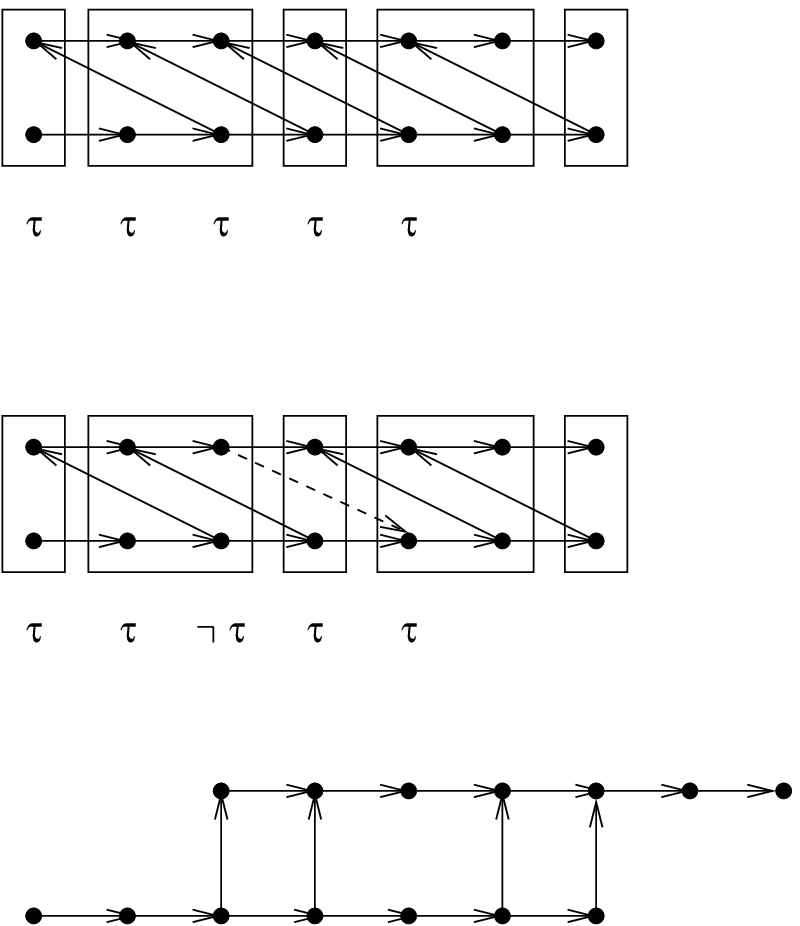}
   \caption{Fixed points that start with series of $\tau$-matches.}
   \label{fig:142536}
  \end{center}
\end{figure}

Note that our arguments above did not use the fact that there 
were $\tau$-matches starting at cells $5, 11, \ldots$.  Indeed, 
these matches are not necessary to force the brick structure described 
above. For example, suppose that there were no $\tau$-match 
starting at cell $5$ but there where $\tau$-matches starting at 
cell 7.  We have pictured this situation on the second 
line of Figure \ref{fig:142536} where we have written 
$\neg \tau$ below the position corresponding to cell 5 to indicate 
that there is not a $\tau$-match starting a cell 5. 
Then one can from the diagram pictured in the second line 
of Figure \ref{fig:142536}, that it must be the case that 
$\sg_6 < \sg_{9}$. It follows that if one looks at the requirements 
on $\sg$ to start with such a series of $\tau$-matches, then 
$\sg$ must be a linear extension of poset whose Hasse diagram is 
pictured at the bottom of Figure \ref{fig:142536}.

There are now two cases depending on where the sequence 
of $\tau$-matches starting at positions $1,3,7,9, \ldots $ ends. \\
\ \\
{\bf Case 2.1.}  There are $\tau$-matches in $\sg$ starting 
at positions $1,3,7,9, \ldots , 6k+3$, but there is no $\tau$-match 
starting at position $6k+7$. This situation is pictured 
in Figure \ref{fig:142536A} in the case where $k=2$. 

In this case, we claim that $\{\sg_1, \ldots, \sg_{6k+8}\} = \{1, 2, \ldots, 6k+8\}$.  If not, then $i$ be the least element in 
$\{1, 2, \ldots, 6k+8\} -\{\sg_1, \ldots, \sg_{6k+8}\}$. The question then 
becomes for which $j$ is $\sg_j =i$. It easy to see from the diagram at 
the top of  Figure \ref{fig:142536A}, that $\sg_{6k+8} > \sg_r$ for 
$r =1, \ldots, 6k+7$. This implies that $\sg_{6k+8} \geq  6k+8$. But since 
$i \in \{1, 2, \ldots, 6k+8\} -\{\sg_1, \ldots, \sg_{6k+8}\}$, it must 
be the case that $\sg_{6k+8} > 6k+8 \geq i$. 

We claim that $j$ cannot equal $6k+9$.  That is, if 
$i = 6k+9$, then $\sg_{6k+8} > \sg_{6k+9}$.  It cannot be 
that $\sg_{6k+8}$ and $\sg_{6k+9}$ are in brick $b_{2k+3}$ because 
then $\sg_{6k+8}$ is labeled with $y$ and there is a $\tau$-match 
contained in bricks $b_{2k+2}$ and $b_{2k+3}$ that ends before 
cell $6k+8$ which means that cell $6k+8$ satisfies Case 1 of our 
definition of $J_{\tau}$ which violates our assumption that $(B,\sg)$ is 
fixed point of $J_{\tau}$.  If $\sg_{6k+9}$ starts brick 
$b_{2k+4}$, then brick $b_{2k+3}$ must be of size 2 and 
there must be a $\tau$-match contained in 
bricks $b_{2k+3}$ and $b_{2k+4}$ that involves  $\sg_{6k+8}$ and $\sg_{6k+9}$. 
But since $\sg_{2k+8}> \sg_{2k+9}$, that $\tau$-match can only start at cell $6k+7$ 
which violates our assumption in this case. 

Next we claim that $j$ cannot be $\geq 6k+10$.  That is, if 
$j \geq  6k+10$, then both $\sg_{j-2}$ and $\sg_{j-1}$ are greater than 
$\sg_j =i$. Thus 
$\sg_{j-1}$ and $\sg_{j}$ must be part of $\tau$-match in $\sg$. 
But then  
the elements in two cells before cell $j$ are bigger than 
that in cell $j$ which means that the only 
role that $\sg_j$ can play in a $\tau$-match
is 1. Thus there can be no $\tau$-match that includes 
$\sg_{j-1}$ and $\sg_{j}$.

\begin{figure}[h]
  \begin{center}
   \includegraphics[width=0.40\textwidth]{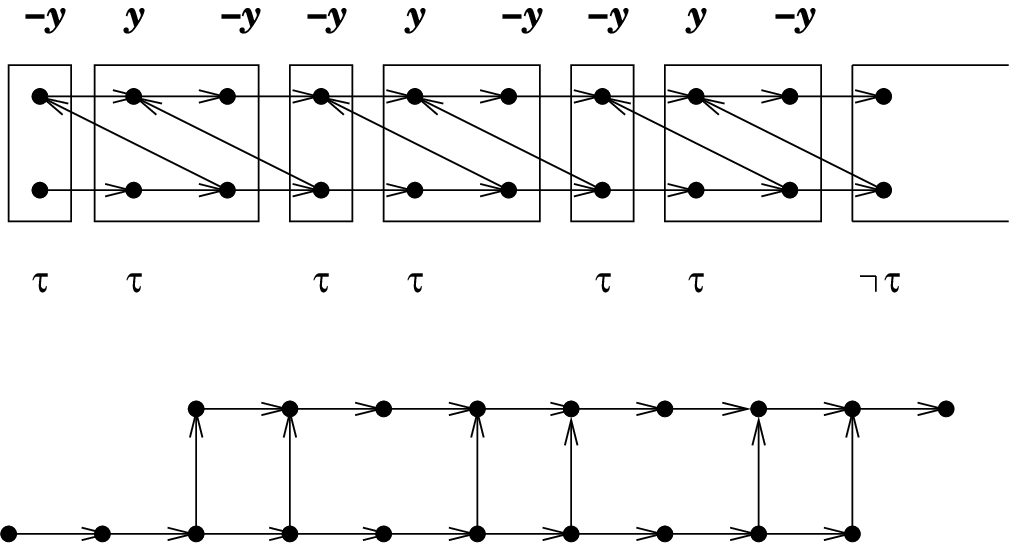}
   \caption{Fixed points that start with series of $\tau$-matches in Case 2.1.}
   \label{fig:142536A}
  \end{center}
\end{figure}

Let $\alpha$ be the permutation 
that is obtained from $\sg$ by removing the elements $1, \ldots, 6k+7$ and 
subtracting $6k+7$ from the remaining elements. Let $B'$ be the 
brick structure $(b_{2k+3}-1,b_{2k+4},\ldots, b_{k})$. Then it is easy to 
see that $(B',\alpha)$ is a fixed point of $J_{\tau}$ is size 
$n-6k-7$. 

Vice versa, suppose we 
start with a fixed point $(B',\alpha)$ of $J_{\tau}$ whose size 
$n-6k-7$ where 
$B' =(d_1,d_2, \ldots, d_s)$. Then we can obtain 
a fixed point $(B,\sg)$ of size $n$ which 
has $\tau$-matches in $\sg$ starting 
at positions $1,3,7,9, \ldots , 6k+3$, but no $\tau$-match 
starting at position $6k+7$ by letting $\sg_1 \ldots \sg_{6k+7}$ be any 
permutation of $1, \ldots, 6k+7$ which is a linear extension of 
the poset whose Hasse diagram is pictured at the bottom of 
Figure \ref{fig:142536A} and letting $\sg_{6k+8} \ldots \sg_n$ be the sequence that 
results by adding  $6k+7$ to each element of $\alpha$.  Then let 
$B=(b_1, \ldots, b_{2k+2}, d_1+1,d_2, \ldots, d_s)$ where $b_{2i+1} =2$ for 
$i =0, \ldots, k$ and $b_{2i}=4$ for $i =1, \ldots, k+1$.

It follows that contribution to $U_{\tau,n}(y)$ from the fixed points in Case 2.1 equal 
$$\sum_{k=0}^{\lfloor \frac{n-8}{6} \rfloor} 
G_{6k+7} y^{3k+3}U_{\tau,n-6k-7},$$
where $G_{6k+7}$ is the number of linear extensions of the poset pictured 
at the bottom of Figure \ref{fig:142536A} of size $6k+7$.

Next we want to compute the number of linear extensions of $G_{6k+7}$. 
It is easy to see that the left-most two elements at the bottom of 
the Hasse diagram of $G_{6k+7}$ must be first two elements of the linear 
extension and the right-most element at the top of the Hasse diagram must be 
the largest element in any linear extension of 
$G_{6k+7}$. Thus the number of linear extensions of $\bar{G}_{6k+4}$ which 
is the Hasse diagram of  $G_{6k+7}$ with those three elements removed, equals 
the number of linear extension of $G_{6k+7}$.  
We have pictured the Hasse diagrams of $\bar{G}_4$, $\bar{G}_{10}$ and $\bar{G}_{16}$ in 
Figure \ref{fig:barG}.

\begin{figure}[h]
  \begin{center}
   \includegraphics[width=0.40\textwidth]{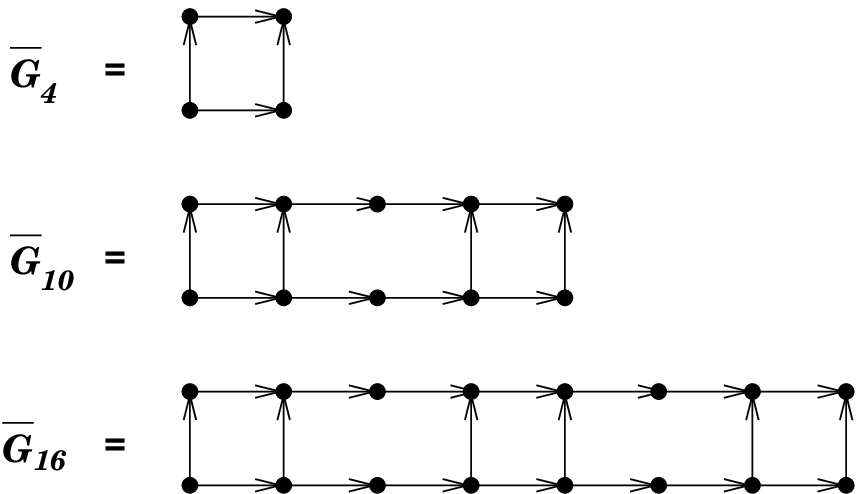}
   \caption{The Hasse diagram of $\bar{G}_{6k+4}$ for $k=0,1,2$.}
   \label{fig:barG}
  \end{center}
\end{figure}

Now let $A_0=1$ and 
$A_{k+1}$ be the number of linear extensions of $\bar{G}_{6k+4}$ for $k \geq 0$. 
It is easy to see that $A_1 =2$.  There is a natural recursion satisfied 
by the $A_k$, namely, for $k > 1$, 
\begin{equation}\label{Arec}
A_{k+1}= \sum_{j=0}^{k} C_{2+3j} A_{k-j}
\end{equation}
where $C_n = \frac{1}{n+1}\binom{2n}{n}$ is the $n$-th Catalan number. 
First, consider the number of linear extensions of the Hasse diagram 
of the poset $D_n$ with $n$ columns of the type pictured in Figure \ref{fig:standard}. 
It is easy to see that this is the number of standard tableaux of 
shape $(n^2)$ which is well known to equal to $C_n$.

\begin{figure}[h]
  \begin{center}
   \includegraphics[width=0.35\textwidth]{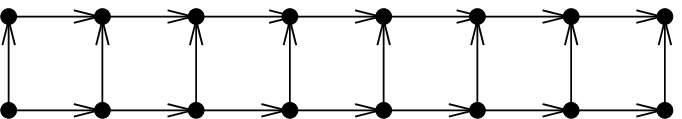}
   \caption{The Hasse diagram of $D_n$.}
   \label{fig:standard}
  \end{center}
\end{figure}

Next if we look at the Hasse diagram of $\bar{G}_{6k+4}$ it is easy to see 
that there are no relation that is forced between 
the elements in columns $3i$ for $i =1, \ldots, k$.  Now suppose 
that we partition the set of linear extensions of  $\bar{G}_{6k+4}$ by 
saying the bottom element in column 3i is less than the top element 
in column $3i$ for $i =1, \ldots, j$ and the top element of column 
$3j+3$ is less than the bottom elements of column $3j+3$.  Then we will have a situation as 
pictured in Figure \ref{fig:partition1} in the case where $k =6$ and $j=2$. 
One can see that when one straightens out the resulting Hasse diagram, it 
starts with the Hasse diagram of $D_{2+3j}$ and all those elements must be 
less than the elements in the top part of Hasse diagram which is a copy of 
the Hasse diagram of $\bar{G}_{6(k-j-1)+4}$.

\begin{figure}[h]
  \begin{center}
   \includegraphics[width=0.60\textwidth]{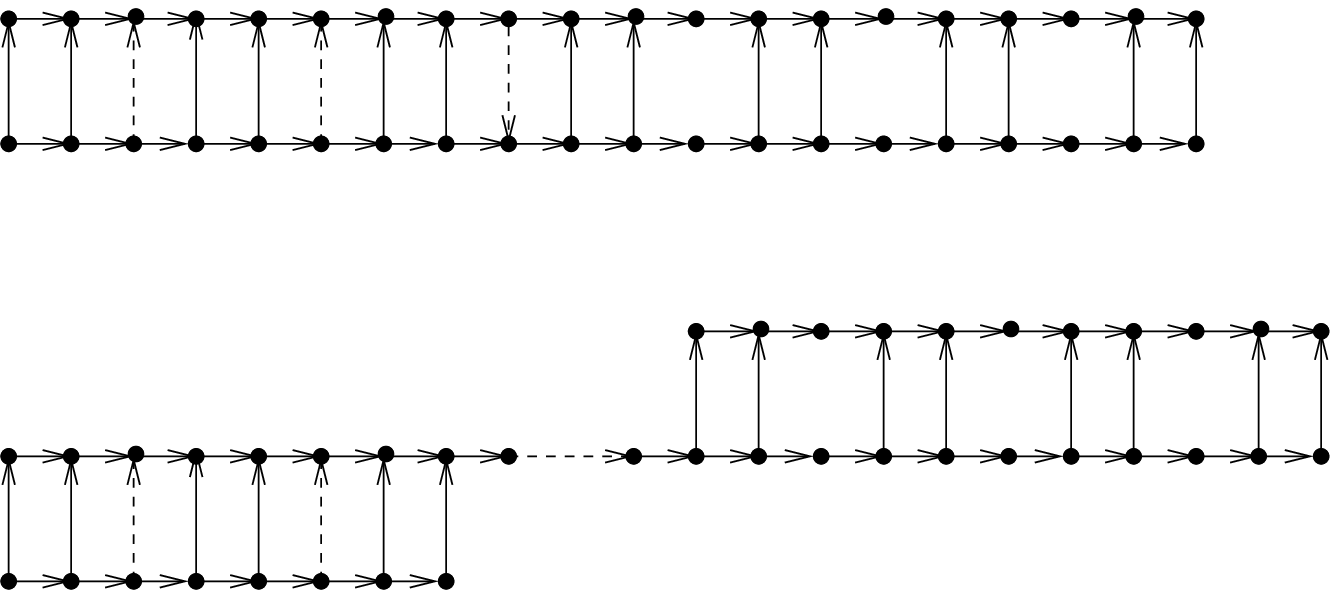}
   \caption{Partitioning the Hasse Diagram of $\bar{G}_{6k+4}$.}
   \label{fig:partition1}
  \end{center}
\end{figure}

Now consider the determinant of the $n \times n$ matrix $M_{n}$ whose elements 
on the main diagonal are $C_2$, the elements on the $j$-diagonal above 
the main are $C_{2+3j}$ for $j \geq 1$, the elements on the sub-diagonal are $-1$, and 
the elements below the sub-diagonal are 0. For example we have pictured 
in $M_7$ in Figure \ref{fig:matrix}. It is then easy to see that 
$\mathrm{det}(M_1) = C_2 =2$. For $n >1$ if we expand the determinant by minors 
about the first row, then we see that we have the recursion 

\begin{equation}\label{Arec}
\mathrm{det}(M_k)= \sum_{j=0}^{k-1} C_{2+3j} \mathrm{det}(M_{k-j-1}),
\end{equation}
where we set $\mathrm{det}(M_0)=1$.

For example, suppose that we expand the determinant $M_7$ pictured 
in Figure \ref{fig:matrix} about the element of 
$C_8$ in the first row. Then in the next two rows, we are forced to 
expand about the $-1$'s.  It is easy to see that the total sign of 
these expansion is always $+1$ so that in this case, we would get a 
contribution of 
$C_8 \mathrm{det}(M_4)$ to $\mathrm{det}(M_7)$.

\begin{figure}[h]
  \begin{center}
   \includegraphics[width=0.25\textwidth]{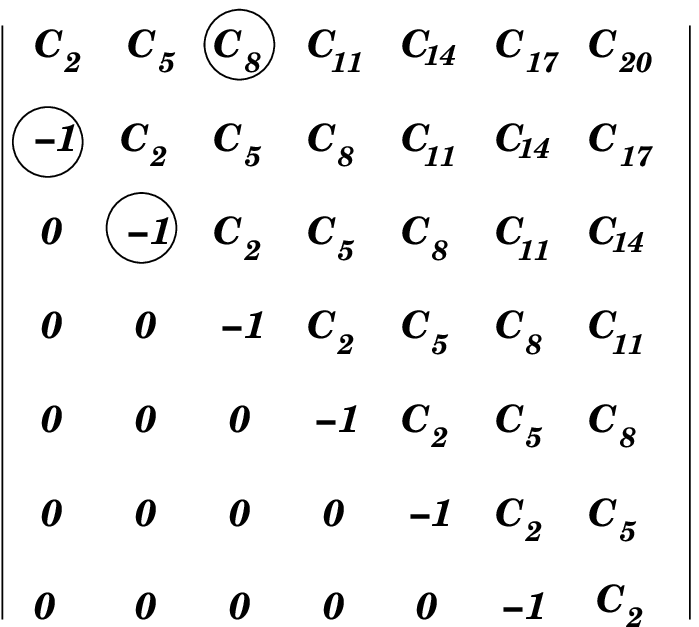}
   \caption{The matrix $M_7$.}
   \label{fig:matrix}
  \end{center}
\end{figure}

Thus it follows that $A_n = \mathrm{det}(M_n)$ for all $n$.

Hence the contribution to $U_{\tau,n}$ from the fixed points in Case 1 equals 
$$\sum_{k=0}^{\lfloor \frac{n-8}{6} \rfloor} 
 \mathrm{det}(M_{k+1}) y^{3k+3}U_{\tau,n-6k-7}.$$
\ \\
{\bf Case 2.2}  There are $\tau$-matches in $\sg$ starting 
at positions $1,3,7,9, \ldots , 6k+1$, but there is no $\tau$-match 
starting at position $6k+3$. This situation is pictured 
in Figure \ref{fig:142536B} in the case where $k=3$.

\begin{figure}[h]
  \begin{center}
   \includegraphics[width=0.40\textwidth]{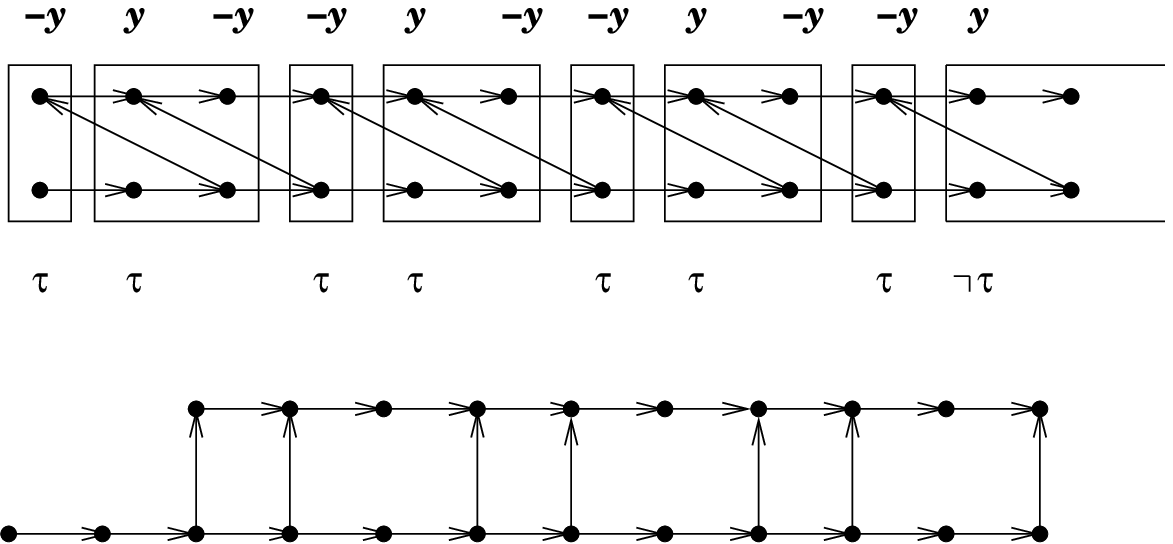}
   \caption{Fixed points that start with series of $\tau$-matches in Case 2.2.}
   \label{fig:142536B}
  \end{center}
\end{figure}

In this case, we claim that $\{\sg_1, \ldots, \sg_{6k+5}\} = \{1, 2, \ldots, 6k+5\}$.  If not, then let $i$ be the least element in 
$\{1, 2, \ldots, 6k+5\} -\{\sg_1, \ldots, \sg_{6k+5}\}$. The question then 
becomes for which $j$ is $\sg_j =i$.  It easy to see from the diagram at 
the top of  Figure \ref{fig:142536B}, that $\sg_{6k+6} > \sg_r$ for 
$r =1, \ldots, 6k+5$ and that $\sg_{6k+5} > \sg_r$ for 
$r =1, \ldots, 6k+5$. This implies that $\sg_{6k+5} \geq 6k+5$, but since 
$i \in \{1, 2, \ldots, 6k+5\} -\{\sg_1, \ldots, \sg_{6k+5}\}$, it follows 
that $6k+5 < \sg_{6k+5} < \sg_{6k+6}$. 
  
It cannot be that $i =\sg_{6k+7}$ because then $\sg_{6k+6}>\sg_{6k+7}$. 
Note that $\sg_{6k+3},\sg_{6k+4},\sg_{6k+5},\sg_{6k+6}$ are elements 
of brick $b_{2k+2}$. If $\sg_{6k+7}$ was also and element of 
brick $b_{2k+2}$, then $\sg_{6k+6}$ would be marked with a $y$ and there 
is a $\tau$-match contained in bricks $b_{2k+1}$ and $b_{2k+2}$ that 
ends at cell $6k+6$ so that we could apply Case 1 of the involution 
$J_{\tau}$ at cell $6k+6$, which violates our assumption 
that $(B,\sg)$ was a fixed point of $J_{\tau}$.  If $\sg_{6k+7}$ starts 
brick $b_{2k+3}$, then there must be a $\tau$-match that involves  
$\sg_{6k+6}$ and $\sg_{6k+7}$ and is contained in bricks $b_{2k+2}$ and $b_{2k+3}$.
Since we are assuming that there is no 
 $\tau$-match cannot starting at $\sg_{6k+3}$, it must be the case that there 
is a $\tau$-match starting at 
$\sg_{6k+5}$.  But then we have that situation pictured in Figure 
\ref{fig:142536BB}. In Figure \ref{fig:142536BB}, the 
 dark arrows are forced by the $\tau$-matches starting 
at $\sg_{6k+1}$ and $\sg_{6k+5}$.  However the top two elements 
in brick $b_{2k+2}$ are $\sg_{6k+5}$ and $\sg_{6k+6}$, which are both 
greater than $i$. This means that the dotted arrow is forced which implies 
that there is a $\tau$-match starting at cell $\sg_{6k+3}$. 

Finally, it cannot be the case that $j > 6k+7$, because then 
it must be the case that $\sg_{j-1} > \sg_j$ so that $\sg_{j-1}$ and 
$\sg_j$ must be part of a $\tau$-match in $\sg$. But in this situation, the 
elements $1, \ldots, i-1$ lie in cells that are more than 2 cells away 
from the cell containing $i$.  This means that in any $\tau$-match in $\sg$ 
containing the element $i$, $i$ can only play the role of 1 in that $\tau$-match.  Thus, there 
could not  be a $\tau$-match containing $\sg_{j-1}$ and $\sg_j$.

\begin{figure}[h]
  \begin{center}
   \includegraphics[width=0.25\textwidth]{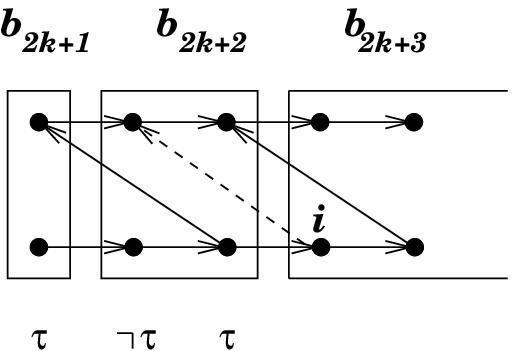}
   \caption{$i$ starts brick $b_{2k+3}$.}
   \label{fig:142536BB}
  \end{center}
\end{figure}

Next, consider the possible $j$ such that $\sg_j = 6k+6$.  It cannot 
be that $j > 6k+7$, because then 
it must be the case that $\sg_{j-1} > \sg_j$ so that $\sg_{j-1}$ and 
$\sg_j$ must be part of a $\tau$-match in $\sg$. But in this situation, the 
elements $1, \ldots, 6k+5$ lie in cells that are more than 2 cells away 
from the cell containing $6k+6$.  This means that in any $\tau$-match containing the 
element $6k+6$ in $\sg$, $6k+6$ can only play the role of 1 in that $\tau$-match.  Thus there 
could not be a  $\tau$-match in $\sg$ containing $\sg_{j-1}$ and $\sg_j$. It follows 
that $6k+6 = \sg_{6k+6}$ or $\sg_{6k+7}$. Let $\alpha$ be the permutation 
that is obtained from $\sg$ by removing the elements $1, \ldots, 6k+4$, setting 
$\alpha_1=1$,  and letting $\alpha_2 \ldots,\alpha_n-(6k+4)$ be the result of 
subtracting $6k+5$ from $\sg_{6k+6} \ldots \sg_n$.  Let $B'$ be the 
brick structure $(b_{2k+2}-2,b_{2k+3},\ldots, b_{k})$. Then it is easy to 
see that $(B',\alpha)$ is a fixed point of $J_{\tau}$ is size 
$n-6k-4$ that starts with a brick of size at least $2$. 

Vice versa, suppose we 
start with a fixed point $(B',\alpha)$ of $J_{\tau}$ whose size 
$n-6k-4$ that starts with a brick of size at least 2 where 
$B' =(d_1,d_2, \ldots, d_s)$. Then we can obtain 
a fixed point $(B,\sg)$ of size $n$ which 
has $\tau$-matches in $\sg$ starting 
at positions $1,3,7,9, \ldots , 6k+1$, but no $\tau$-match 
starting at position $6k+3$, by letting $\sg_1 \ldots \sg_{6k+5}$ be any 
permutation of $1, \ldots, 6k+5$ which is a linear extension of 
the poset whose Hasse diagram is pictured at the bottom of 
Figure \ref{fig:142536B} and letting $\sg_{6k+6} \ldots \sg_n$ be the sequence that 
results by adding  $6k+5$ to each element of $\alpha_2 \ldots \alpha_{n-(6k+4)}$.  We let 
$B=(b_1, \ldots, b_{2k+1}, d_1+2,d_2, \ldots, d_s)$ where $b_{2i+1} =2$ for 
$i =0, \ldots, k$ and $b_{2k}=4$ for $i =1, \ldots, k$.

Note that for any $n$, our arguments above show that 
the only fixed points $(D,\gamma)$ of $J_{\tau}$ of size $n$ where $D=(d_1, \ldots, 
d_k)$ and $\sg = \sg_1 \ldots \sg_n$ which do not start with a brick of size at least $2$ are 
the ones that start with a brick $b_1=1$ where $\sg_1 =1$ and $\sg_2 =2$. Clearly 
such fixed points are counted by $-yU_{n-1,y}$ because $d_1$ would have weight $-y$ 
and $((d_2, \ldots, d_k),(\sg_2-1) (\sg_3-1) \ldots (\sg_n-1))$ could be any 
fixed point of $J_\tau$ of size $n-1$. It follows that sum of the weights of 
all fixed points of $J_\tau$ of size $n$ which start with a brick of size at least 2 is 
equal to 
$$U_{\tau,n}-(-yU_{n-1,\tau}) =U_{\tau,n}+yU_{n-1,\tau}.$$

It follows that contribution to $U_{\tau,n}$ from the fixed points in Case 2.2 equal 
$$-\sum_{k=0}^{\lfloor \frac{n-6}{6} \rfloor} 
G_{6k+4}y^{3k+2}(U_{\tau,n-6k-4}+yU_{\tau,n-6k-5}),$$
where $G_{6k+4}$ is the number of linear extensions of the poset pictured 
at the bottom of Figure \ref{fig:142536B} of size $6k+4$.

Next we want to compute the number of linear extensions of $G_{6k+4}$. 
It is easy to see that the left-most two elements at the bottom of 
the Hasse diagram of $G_{6k+4}$ must be first two elements of the linear 
extension. Thus the number of linear extensions of $\bar{G}_{6k+2}$ which 
is the Hasse diagram of  $G_{6k+4}$ with those two elements removed, equals 
the number of linear extension of $G_{6k+4}$. 
We have pictured the Hasse diagrams of $\bar{G}_2$, $\bar{G}_{8}$ and $\bar{G}_{14}$ in 
Figure \ref{fig:barG2}.

\begin{figure}[h]
  \begin{center}
   \includegraphics[width=0.40\textwidth]{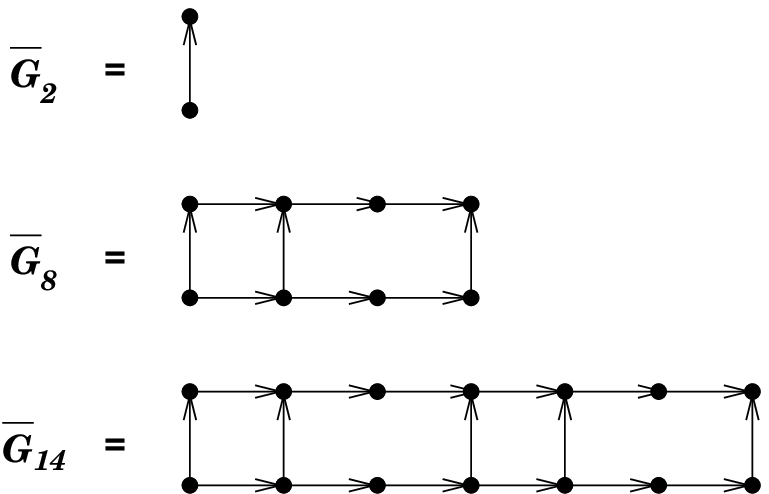}
   \caption{The Hasse diagram of $\bar{G}_{6k+2}$ for $k=0,1,2$.}
   \label{fig:barG2}
  \end{center}
\end{figure}

Now let $B_0=1$ and 
$B_{k+1}$ be the number of linear extensions of $\bar{G}_{6k+2}$ for $k \geq 0$. 
It is easy to see that $B_1 =1$.  Again there is a natural recursion satisfied 
by the $B_k$s, namely, for $k > 1$, 
\begin{equation}\label{Brec}
B_{k+1}=  C_{3k+1} +\sum_{j=0}^{k-1} C_{2+3j} B_{k-j-1},
\end{equation}
where $C_n = \frac{1}{n+1}\binom{2n}{n}$ is the $n$-th Catalan number. 

As in the case of the posets $\bar{G}_{6k+4}$, 
there is no relations that is forced between the elements of 
the elements in columns $3i$ for $i =1, \ldots, k$.  Now suppose 
that we partition the set of linear extensions of  $\bar{G}_{6k+2}$ by 
saying the bottom element in column 3i is less than the top element 
in column $3i$ for $i =1, \ldots, j$ and the top element of column 
$3j+3$ is less than the bottom elements of column $3j+3$.  First if $j=k$, 
then we will have a copy of $D_{3k+1}$ which gives a contribution 
of $C_{3k+1}$ to the number of linear extensions of $\bar{G}_{6k+4}$. 
If $j < k$, then we will have a situation as 
pictured in Figure \ref{fig:partition2} in the case where $k =6$ and $j=2$. 
One can see that when one straightens out the resulting Hasse diagram, one obtains a diagram that 
starts with the Hasse diagram of $D_{2+3j}$ and all those elements must be 
less than the elements in the top part of Hasse diagram which is a copy of 
the Hasse diagram of $\bar{G}_{6(k-j-1)+2}$.

\begin{figure}[h]
  \begin{center}
   \includegraphics[width=0.40\textwidth]{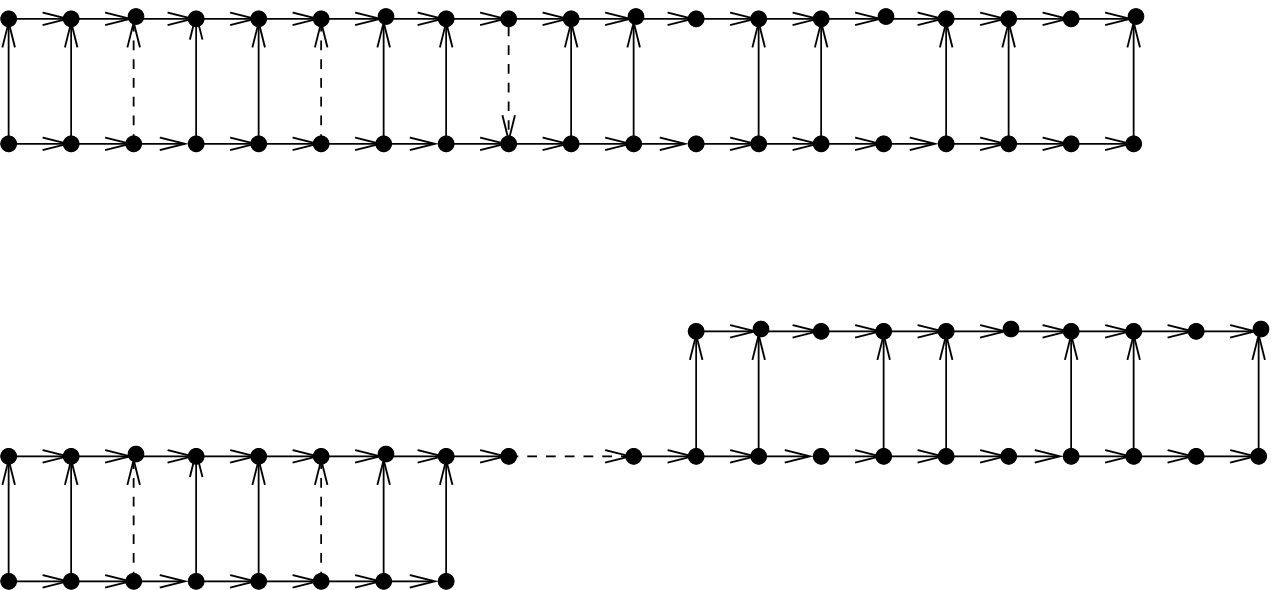}
   \caption{Partitioning the Hasse Diagram of $\bar{G}_{6k+2}$.}
   \label{fig:partition2}
  \end{center}
\end{figure}

Let $P_n$ be the matrix that is obtained from the matrix $M_n$ by 
replacing the elements $C_m$ in the last 
column by $C_{m-1}$. For example we have pictured 
in $P_7$ in Figure \ref{fig:matrix2}. It is then easy to see that 
$\mathrm{det}(P_1) = 1$. For $n >1$ if we expand the determinant by minors 
about the first row, then we see that we have the recursion 

\begin{equation}\label{Arec}
\mathrm{det}(P_k)= C_{3k-2}+\sum_{j=0}^{k-2} C_{2+3j} \mathrm{det}(P_{k-j-1}),
\end{equation}
where we set $\mathrm{det}(P_0)=1$.

For example, suppose that we expand the determinant $P_7$ pictured 
in Figure \ref{fig:matrix2} about the element of 
$C_{19}$ in the first row. Then in the next five rows, we would be forced to 
expand about the $-1$'s.  It is easy to see that the total sign of 
these expansion is always $+1$ so that in this case, we would get a 
contribution of 
$C_{19}$ to the $\mathrm{det}(P_7)$. Expanding the determinant about the 
other elements in the first row gives the remaining terms of the recursion 
just like it did in the expansion of the determinant of $M_n$.

\begin{figure}[h]
  \begin{center}
   \includegraphics[width=0.25\textwidth]{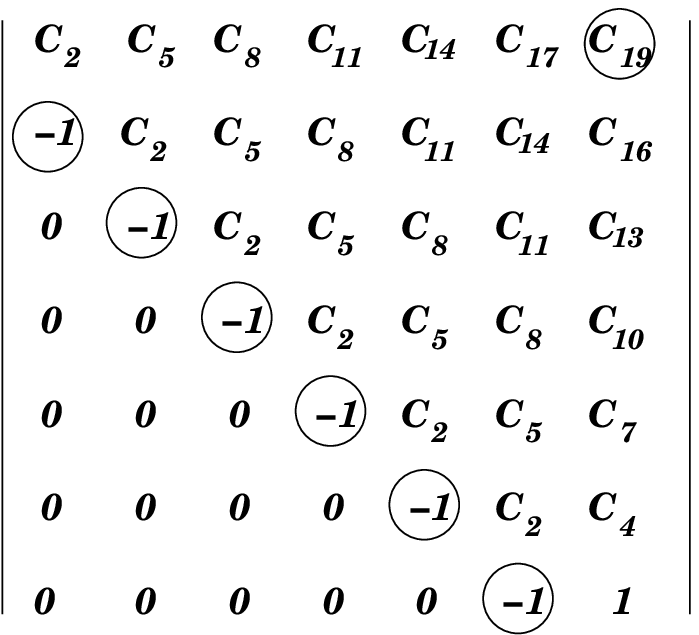}
   \caption{The matrix $P_7$.}
   \label{fig:matrix2}
  \end{center}
\end{figure}

Thus it follows that $B_n = \mathrm{det}(P_n)$ for all $n$.

Hence the contribution of fixed points of $J_\tau$ to $U_{\tau,n}(y)$ 
in the Case 2.2 equals 
$$-\sum_{k=0}^{\lfloor \frac{n-6}{6} \rfloor} 
\mathrm{det}(P_{k+1}) y^{3k+2}(U_{\tau,n-6k-4}+yU_{\tau,n-6k-5}).$$

Therefore, we obtain the recursion for $U_{\tau,n}(y)$ for $\tau = 142536$ is as follows.
\begin{align*}
\displaystyle U_{\tau,n}(y) = &~~~  (1-y)U_{\tau,n-1}(y) +  
\sum_{k=0}^{\lfloor(n-8)/6 \rfloor}
\mathrm{det}(M_{k+1}) y^{3k+3} U_{\tau,n-6k-7}(y) \\
& ~~~\qquad  - \sum_{k=0}^{\Floor[(n-6)/6]} \mathrm{det}(P_{k+1})y^{3k+2}
\left[U_{\tau,n-6k-4}(y) +yU_{\tau,n-6k-5}(y) \right].
\end{align*}

In Table \ref{tab:U142536}, we computed $U_{142536,n}(y)$ for $n \leq 14$.  

\begin{table}[ht]
\begin{center}
\begin{tabular}{c|l}
n & $U_{142536,n}(y)$ \\ 
\hline 1 & $-y$\\
2 & $-y + y^2$ \\
3 & $-y + 2y^2 - y^3$\\
4 & $-y + 3 y^2 - 3 y^3 + y^4$\\
5 & $-y + 4 y^2 - 6 y^3 + 4 y^4 - y^5$\\
6 & $-y + 5y^2 - 9y^3 + 10y^4 -5y^5 + y^6$ \\ 
7 & $-y + 6 y^2 - 13 y^3 + 18 y^4 - 15 y^5 + 6 y^6 - y^7$ \\
8 & $-y + 7 y^2 - 18 y^3 + 27 y^4 - 32 y^5 + 21 y^6 - 7 y^7 + y^8$ \\ 
9 & $-y + 8 y^2 - 24 y^3 + 40 y^4 - 54 y^5 + 52 y^6 - 28 y^7 + 8 y^8 - y^9$  \\ 
10 & $-y + 9 y^2 - 31 y^3 + 58 y^4 - 85 y^5 + 100 y^6 - 79 y^7 + 36 y^8 - 9 y^9 + y^{10}$  \\ 
11 & $-y + 10 y^2 - 39 y^3 +82 y^4 - 129 y^5 + 170 y^6 - 172 y^7 + 
 114 y^8 - 45 y^9 + 10y^{10} - y^{11}$  \\ 
12 & $-y + 11 y^2 - 48 y^3 + 113 y^4 - 191 y^5 + 289 y^6 - 320 y^7 + 278 y^8$ \\
 &\qquad \qquad \qquad \qquad \qquad \qquad \qquad  \qquad \qquad $- 158 y^9 + 55 y^{10} - 11 y^{11} + y^{12}   $\\
13 & $-y + 12 y^2 - 58 y^3 + 152 y^4 - 277 y^5 + 456 y^6 - 578 y^7 + 568 y^8 - 427 y^9$\\
 & \qquad \qquad \qquad \qquad \qquad  \qquad \qquad \qquad $  +212 y^{10} - 66 y^{11} + 12 y^{12} - y^{13}$\\
14 & $-y + 13 y^2 - 69 y^3 + 200 y^4 - 394 y^5 + 689 y^6 - 1031 y^7 + 
 1068 y^8 + 956 y^9  $  \\
  & \qquad \qquad \qquad  \qquad \qquad \qquad \qquad $+ 629 y^{10} - 277 y^{11} + 78 y^{12} - 13 y^{13} + y^{14}  $ 
\end{tabular} 
\end{center}
\caption{The polynomials $U_{\tau,n}(y)$ for $\tau = 142536.$} \label{tab:U142536}
\end{table}

\section{The proof of Theorem \ref{thm:162534}  }\label{sec:162534}

Let $\tau_a = \tau = \tau_1 \ldots, \tau_{2a}$ where 
$\tau_1, \tau_3, \ldots, \tau_{2a-1} = 12 \ldots a$ and 
$\tau_2 \tau_4 \ldots\tau_{2a} = (2a) (2a-1) \ldots (a+1)$. If we picture 
$\tau_a$ in a 2-line array like we did in the last section, then we will get a diagram 
as pictured in Figure \ref{fig:2a}

\begin{figure}[h]
  \begin{center}
   \includegraphics[width=0.40\textwidth]{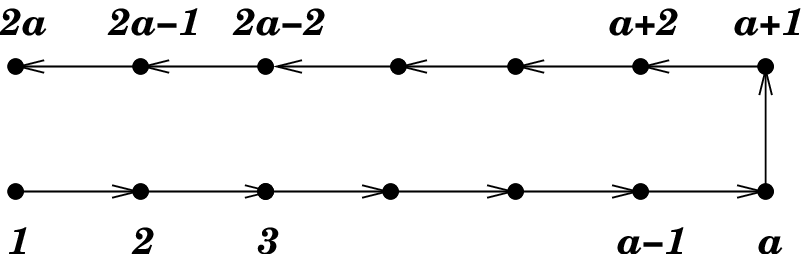}
   \caption{The Hasse diagram associated with $\tau_a$.}
   \label{fig:2a}
  \end{center}
\end{figure}

The key property that $\tau_a$ has is that if $\sg = \sg_1 \ldots \sg_{2m}$ is 
permutation where we have marked some of the $\tau_a$-matches 
by placing an $x$ at the start of a $\tau$ so that every element 
of $\sg$ is contained in some $\tau_a$-match and any two consecutive marked 
$\tau_a$ in $\sg$ share at least one element, then it must 
be the case that 
$\sg_1\sg_3 \ldots \sg_{2m-1} = 12 \ldots m$ and $\sg_2\sg_4 \ldots \sg_{2m} = 
(2m) (2m-1) \ldots (m+1)$. That is, it must be the case that 
$\sg = \tau_{m}$.  This can easily be seen from the picture of 
overlapping $\tau_a$-matches like the one pictured in 
Figure \ref{fig:2acluster} where $a =4$ and $m=12$. Note that 
in such a situation, we will in fact have $\tau_a$ matches starting 
at positions $1,3,5, \ldots, 2(m-a)+2$ in $\sg$.

\begin{figure}[h]
  \begin{center}
   \includegraphics[width=0.40\textwidth]{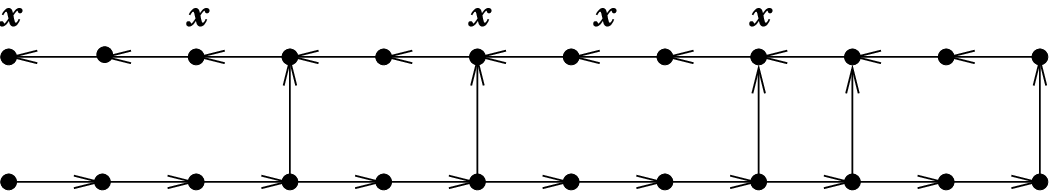}
   \caption{The Hasse diagram of overlapping $\tau_a$-mathces.}
   \label{fig:2acluster}
  \end{center}
\end{figure}

We need to show that the polynomials $$U_{\tau_a,n }(y) = \sum_{O \in \mathcal{O}_{\tau_a,n}, J_{\tau_a}(O) =O} \sgn{O} W(O)$$ satisfy the following properties: \begin{enumerate}
\item $U_{\tau,1}(y)=-y$, and 
\item for $n \geq 2,$ 
\begin{align*}
U_{\tau,n}(y)  = & ~  (1-y)U_{\tau,n-1}(y)  - 
\sum_{k=0}^{\lfloor (n-2a)/(2a)\rfloor} \binom{n-(k+1)a-1}{(k+1)a-1} y^{(k+1)a-1}
U_{\tau_a,n-(2(k+1)a)+1}(y) \\
& \qquad \quad +\sum_{k=0}^{\lfloor (n-2a-2)/(2a)\rfloor} \binom{n-(k+1)a-2}{(k+1)a} 
y^{(k+1)a} U_{\tau_a,n-(2(k+1)a)-1}(y).     
\end{align*}
\end{enumerate}  

Again, it is easy to see that when $n = 1, U_{\tau_a,1}(y)=-y.$ For $n \geq 2,$ let $O=(B,\sg)$ be a fixed point of $J_{\tau_{a}}$ where $B=(b_1, \ldots, b_t)$ and $\sg=\sg_1 \cdots \sg_n$. By the same argument as the previous sections, it must be the case that 1 is in the first cell of $O$ and 2 must be in either cell of 2 or cell 3 in $O.$ Thus, we now have two cases. \\
\ \\
{\bf Case 1.} 2 is in cell 2 of $O$.\\
\ \\
Similar to Case 1 in the proof of Theorem \ref{thm:142536}, there are two possibilities, namely, either (i) 1 and 2 are both in the first brick $b_1$ of $(B,\sg)$ or (ii) brick $b_1$ is a single cell filled with 1 and 2 is in the first cell of the second brick $b_2$ of $O$.  In either case, we can remove cell 1 from $O$ and subtract $1$ from the elements in the remaining cells, we will obtain a fixed point $O'$ of $J_{\tau_a}$ in $\mathcal{O}_{\tau_a,n-1}.$ So the fixed points in this case will contribute $(1-y)U_{\tau_a,n-1}(y)$ to  $U_{\tau_a,n}(y)$.\\
\ \\
{\bf Case 2.} 2 is in cell 3 of $O=(B,\sg)$. \\
\ \\
In this case, $\sg_2 > \sg_3 =2$. Since $\sg$ must be increasing in $b_1$, it 
follows that 2 is in the first cell of brick $b_2$ and there must be a $\tau_a$ 
match in the cells of $b_1$ and $b_2$ which can only start at cell 1. 
Thus it must be the case that brick $b_2$ has at least $2a-2$ cells. 

Again, we shall think of $O=(B,\sg)$ as a two line array $A(0)$ where 
column $i$ consists of $\sg_{2i-1}$ and $\sg_{2i}$, reading from bottom 
to top. Now imagine that $A(0)$ starts with series of $\tau$-matches starting 
at positions $1,3,5, \ldots$. Our observation above shows that if this sequence 
of consecutive $\tau_a$-matches covers cells $1, \ldots ,2k$ for some 
$k$, then in the two line array $A(O)$, all in entries in the first 
row of the first $k$ columns are less than all the entries in top row 
of the first $k$ columns, the cells in the bottom row of the first 
$k$ columns are increasing, reading from left to right, and the cells 
in top row are increasing, reading from right to left.

Next we consider the possible brick structures of $O=(B,\sg)$. We claim 
that we are in one of two subcases:  Subcase (2.A) where there is a  $k \geq 0$ such 
that there are $\tau_a$-matches in $\sg$ starting at cells 
$1, 3, 2a+1, 2a+3, \ldots ,2(k-1)a+1,2(k-1)a+3,2ka+1$, there is no $\tau_a$-match
in $\sg$ starting at cell $2ka+3$, $2 =b_1 =b_3 = \cdots = b_{2k-1}$, 
$2a-2 = b_2 =b_4 = \cdots = b_{2k}$, and $b_{2k+1} =2$ and $b_{2k+2} \geq 2a-2$ or 
Subcase (2.B) where there is a  $k \geq 0$ such 
that there are $\tau_a$-matches in $\sg$ starting at cells 
$1, 3, 2a+1, 2a+3, \ldots ,2(k-1)a+1,2(k-1)a+3,2ka+1,2ka+3$,  there is no $\tau_a$-match
in $\sg$ starting at cell $2(k+1)a+1$, $2 =b_1 =b_3 = \cdots = b_{2k-1}=b_{2k+1}$, 
$2a-2 = b_2 =b_4 = \cdots = b_{2k+2}$, and $b_{2k+3} \geq 2$. Subcase (2.A) is pictured 
at the top of Figure \ref{fig:2taua} and Subcase (2.B) is pictured at the 
bottom of Figure \ref{fig:2taua} in the case where $a=4$ and $k=2$. 
Note that by our remarks above, we also 
know the relative order of the elements involved in these $\tau_a$-matches in $\sg$ 
which is indicated by the poset whose Hasse diagram is pictured in Figure \ref{fig:2taua}.
We can prove this by induction. That is, suppose $k =0$ and we are in Subcase (2.A). 
Then there is a $\tau_a$-match in $\sg$ 
starting a cell 1 but no $\tau_a$-match in $\sg$ starting at cell 3. 
Our argument above shows that $b_1=2$ and $b_2 \geq 2a-2$. 
Next suppose that $k =0$ and we are in Subcase (2.B)
so that there are $\tau_a$-matches in $\sg$ starting in cells 1 and 3 but there 
is no $\tau_a$-match in $\sg$ starting at cell $2a+1$. Then we claim 
we claim that $b_2 =2a-2$.  That is, in such a situation 
we would know that $\sg_{2a} > \sg_{2a+1}$. Thus, 
if $b_2 > 2a-2$, then $2a$ would be labeled with a $y$. The $\tau_{a}$-match 
starting at cell 1 ends at cell $2a$ so that cell $2a$ would satisfy 
Case I of our definition of $J_{\tau_a}$ which 
contracts that the fact that  $O=(B,\sg)$ is a fixed point of $J_{\tau_a}$.  
Thus, brick $b_3$ must start at cell $2a+1$. 
Now the fact that $\sg_{2a} > \sg_{2a+1}$ implies that $b_3 \geq 2$ since 
there must be a $\tau_{a}$-match that involves $\sg_{2a}$ and $\sg_{2a+1}$ and lies in cells of $b_2$ and $b_3$.

\begin{figure}[h]
  \begin{center}
   \includegraphics[width=0.40\textwidth]{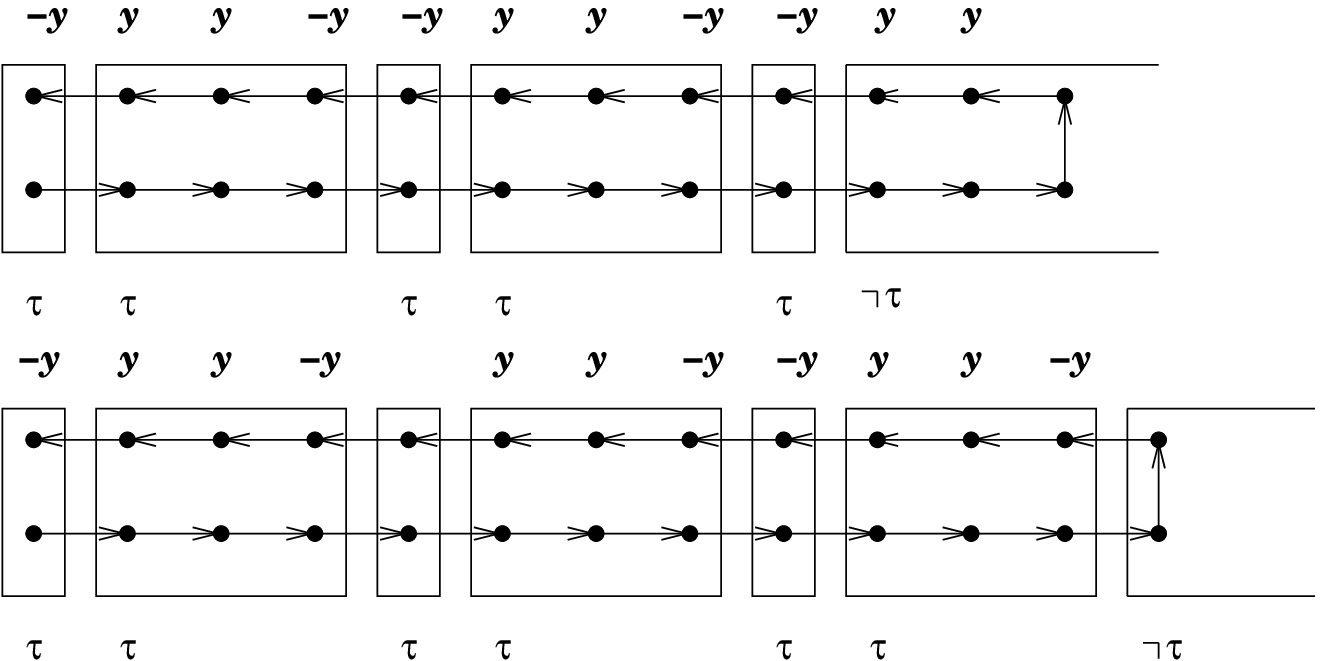}
   \caption{Subcases (2.A) and (2.B).}
   \label{fig:2taua}
  \end{center}
\end{figure}

Now assume by induction that for $k \geq 1$, 
there are $\tau_a$-matches in $\sg$ starting at cells 
$1, 3, 2a+1, 2a+3, \ldots ,2(k-1)a+1,2(k-1)a+3$, $2 =b_1 =b_3 = \cdots = b_{2k-1}$,
$2a-2 = b_2 =b_4 = \cdots = b_{2k-2}$, and $b_{2k} \geq 2a-2$. Suppose we are in Subcase (2.A) so that there is $\tau_a$-match 
starting at cell $2ka+1$ but there is no $\tau_a$ starting at cell $2ka+3$.  
Then we know that $\sg_{2ka} > \sg_{2ka+1}$ due to the $\tau_a$-match 
in $\sg$ starting at cell $2(k-1)a+1$. It cannot be the case that $b_{2k} > 2a-2$ since 
then cells $2ka$ and $2ka+1$ are contained in brick $b_{2k}$ so that cell $2ka$ would 
be marked with a $y$.  However, the $\tau_a$-match staring at cell $2(k-1)a+1$, which is 
the first cell of $b_{2k}$, ends at cell $2ka$ so that cell $2ka$ would 
satisfy Case I of our definition of $J_{\tau_a}$ which violates our assumption 
that $(B,\sg)$ is a fixed point of $J_{\tau_a}$.  This means 
that $b_{2k}=2a-2$ and $b_{2k+1}$ starts at cell $2ka+1$. Since $\sg_{2ak} > \sg_{2ak+1}$ 
due to the $\tau_a$-match in $\sg$ starting at cell $2(k-1)a+3$, we know 
that there must be a $\tau_a$-match contained in the cells of $b_{2k}$ and $b_{2k+1}$ 
so that $b_{2k+1} \geq 2$. But then because of the $\tau_a$-match in $\sg$ starting 
at cell $2ka+1$, we know that $\sg_{2ka+2} > \sg_{2ka+3}$. It cannot be 
that cell $2ka+3$ is in brick $b_{2k+1}$ because then cell $2k+2$ would be marked 
with a $y$ and there is a $\tau_a$-match in $\sg$ starting at cell $2(k-1)a+3$ which 
ends at cell $2k+2$ which is contained in the bricks $b_{2k}$ and $b_{2k+1}$ which 
means that cell $2ka+2$ would 
satisfy Case 1 of our definition of $J_{\tau_a}$ which violates our assumption 
that $(B,\sg)$ is a fixed point of $J_{\tau_a}$. Thus it must be the case 
that $b_{2k+1} =2$ and brick $b_{2k+2}$ starts at cell $2ka+3$. But this 
means that there must be a $\tau_a$-match in $\sg$ contained in the cells 
of $b_{2k+1}$ and $b_{2k+2}$ so that $b_{2k+2} \geq 2a-2$. 
Now if there is also a $\tau_a$-match in $\sg$ starting at cell $2ka+3$, 
then we claim that $b_{2k+2}=2a-2$.  That is, we know 
that $\sg_{2(k+1)a} > \sg_{2(k+1)a+1}$. It cannot be that  $b_{2k+2} > 2a-2$ because 
then cell $2(k+1)a$ would be labeled with a $y$ and the $\tau_a$-match in $\sg$ starting 
at cell $2ka+1$ ends at cell $2(k+1)a$ and is contained in the bricks 
$b_{2k+1}$ and $b_{2k+2}$ so that cell $2(k+1)a$ would satisfy Case 1 of our 
definition of $J_{\tau_a}$ which would violate our assumption that 
$(B,\sg)$ is fixed point of $J_{\tau_a}$.  Thus $b_{2k+2}=2a-2$. But 
then due to the $\tau_a$-match in $\sg$ starting at cell 
$2(k+1)a+3$, we know that $\sg_{2(k+1)a} > \sg_{2(k+1)a+1}$ which means 
that there must be a $\tau_a$ match contained in bricks $b_{2k+2}$ and $b_{2k+3}$. 
This means that $b_{2k+3} \geq 2$.

Thus we have two cases to consider. \\
\ \\
{\bf Subcase (2.A)} There is a  $k \geq 0$ such 
that there are $\tau_a$-matches in $\sg$ starting at cells 
$1, 3, 2a+1, 2a+3, \ldots ,2(k-1)a+1,2(k-1)a+3,2ka+1$, there is no $\tau_a$-match
in $\sg$ starting at cell $2ka+3$, $2 =b_1 =b_3 = \cdots = b_{2k-1}$, 
$2a-2 = b_2 =b_4 = \cdots = b_{2k}$, and $b_{2k+1} =2$ and $b_{2k+2} \geq 2a-2$.

In this case, we claim that $\{1, \ldots, (k+1)a+1\} = \{\sg_1, \sg_3, \ldots, \sg_{2(k+1)a-1},\sg_{2(k+1)a}\}$.  That is, if one considers the diagram at the top of 
Figure \ref{fig:2taua}, then the elements in the bottom row are 
$1,2, \ldots, (k+1)a$, reading from left to right, and the element 
at the top of column $(k+1)a$ is equal to $(k+1)a+1$. 
If this is not the case, then let 
$$i = \min(\{1, \ldots, (k+1)a+1\} - \{\sg_1, \sg_3, \ldots, \sg_{2(k+1)a-1},\sg_{2(k+1)a}\}).$$This means $\sg_{2(k+1)a} > i$ and, hence one can see by the relative order of the elements 
in the first $(k+1)a$ columns of $A(O)$ that $i$ can not lie in the first 
$(k+1)a$ columns. Then the question is for what $j$ is $\sg_j=i$.  First we claim that 
it  cannot be that $\sg_{2(k+1)a+1} =i$. That is, in such a  situation,
$\sg_{2(k+1)a}> \sg_{2(k+1)a+1}$.  Now it cannot be that 
$\sg_{2(k+1)a}$ and $\sg_{2(k+1)a+1}$ lie in brick $b_{2k+2}$ because 
then the $\tau_a$-match in $\sg$ that starts in 
the first cell of $b_{2k+1}$ ends at cell $2(k+1)a$ which means that 
cell $2(k+1)a$ would be labeled with a $y$ and satisfy Case I of our 
definition of $J_{\tau_a}$ which would violate our assumption that 
$(B,\sg)$ is fixed point of $J_{\tau_a}$. Thus it must be the case 
that brick $b_{2k+3}$ starts at cell $2(k+1)a+1$. But then there 
must be a $\tau_a$-match in $\sg$ contained in the cells of bricks 
$b_{2k+2}$ and $b_{2k+3}$ which would imply that there is a $\tau_a$-match 
in $\sg$ starting at cell  $2ka+3$ which violates our assumption in this case. 
Hence $j > 2(k+1)a+1$ which implies  
that both $\sg_{j-2}$ and $\sg_{j-1}$ are greater than $\sg_j =i$. But then 
there could be no $\tau_a$-match in $\sg$ which contains both 
$\sg_{j-1}$ and $\sg_j$ because the only role that $i$ could play in 
$\tau_a$-match in $\sg$ would be 1 under those circumstances. 

It follows that if we remove the elements in $A(0)$ from the first 
$(k+1)a-1$ columns plus the bottom element of column $(k+1)a$, then 
$(B',\sg')$, where $B'=(b_{2k+2}-(2a-1),b_{2k+3}, \ldots, b_t)$ 
and $\sg' = \mathrm{red}(\sg_{2(k+1)a} \ldots \sg_n)$, will 
be a fixed point of $J_{\tau_a}$ of size $n-(2(k+1)a)+1$.  Note that 
in such a situation, we will have 
$\binom{n-(k+1)a-1}{(k+1)a-1}$ ways to choose the elements of 
that lie in the top rows of the first $(k+1)a-1$ columns of $A(O)$. 
Note that the powers of $y$ coming from the bricks $b_1, \ldots, b_{2k}$ is 
$y^{ka}$ and the powers of $y$ coming from bricks $b_{2k+1}$ and $b_{2k+2}$ 
is $-y^{a-1}$. It follows 
that the elements in Subcase (2.A) contribute 
$$-\sum_{k=0}^{\lfloor (n-2a)/(2a)\rfloor} \binom{n-(k+1)a-1}{(k+1)a-1} 
y^{(k+1)a-1} U_{\tau_a,n-(2(k+1)a)+1}(y)$$
to $U_{\tau_a,n}(y)$. \\
\ \\
{\bf Subcase (2.B).} There is a  $k \geq 0$ such 
that there are $\tau_a$-matches in $\sg$ starting at cells 
$1, 3, 2a+1, 2a+3, \ldots ,2(k-1)a+1,2(k-1)a+3,2ka+1,2ka+3$, there is no $\tau_a$-match
in $\sg$ starting at cell $2(k+1)a+1$, $2 =b_1 =b_3 = \cdots = b_{2k-1}=b_{2k+1}$, 
$2a-2 = b_2 =b_4 = \cdots = b_{2k+2}$, and $b_{2k+3} \geq 2$.

In this case, we claim that $\{1, \ldots, (k+1)a+2\} = \{\sg_1, \sg_3, \ldots, \sg_{2(k+1)a+1},\sg_{2(k+1)a+2}\}$.  That is, if one considers the diagram at the bottom of 
Figure \ref{fig:2taua}, then the elements in the bottom row are 
$1,2, \ldots, (k+1)a+1$, reading from left to right, and the element 
at the top of column $(k+1)a+1$ is equal to $(k+1)a+2$. 
If this is not the case, then let 
$$i = \min(\{1, \ldots, (k+1)a+2\} - \{\sg_1, \sg_3, \ldots, \sg_{2(k+1)a+1},\sg_{2(k+1)a+2}\}).$$
This means $\sg_{2(k+1)a+2} > i$ and, hence one can see by the relative order of the elements 
in the first $(k+1)a+1$ columns of $A(O)$ that $i$ can not lie in the first 
$(k+1)a+1$ columns. Then the question is for what $j$ is $\sg_j=i$.  First we claim that 
it  cannot be that $\sg_{2(k+1)a+3} =i$. That is, in such as situation,
$\sg_{2(k+1)a+2}> \sg_{2(k+1)a+3}$.  Now it cannot be that 
$\sg_{2(k+1)a+2}$ and $\sg_{2(k+1)a+3}$ lie in brick $b_{2k+3}$ because 
then the $\tau_a$-match in $\sg$ that starts in 
the first cell of $b_{2k+2}$ ends at cell $2(k+1)a+2$ which means that 
cell $2(k+1)a+2$ would be labeled with a $y$ and satisfy Case I of our 
definition of $J_{\tau_a}$ which would violate our assumption that 
$(B,\sg)$ is fixed point of $J_{\tau_a}$. Thus it must be the case $b_{2k+3}=2$ 
that brick $b_{2k+4}$ starts at cell $2(k+1)a+3$. But then there 
must be a $\tau_a$-match in $\sg$ contained in the cells of bricks 
$b_{2k+3}$ and $b_{2k+4}$ which would imply that there is a $\tau_a$-match 
in $\sg$ starting at cell  $2(k+1)a+1$ which violates our assumption in this case. 
Hence $j > 2(k+1)a+3$ which implies  
that both $\sg_{j-2}$ and $\sg_{j-1}$ are greater than $\sg_j =i$. But then 
there could be no $\tau_a$-match in $\sg$ which contains both 
$\sg_{j-1}$ and $\sg_j$ because the only role that $i$ could play in 
$\tau_a$-match in $\sg$ would be 1 under those circumstances. 

It follows that if we remove the elements in $A(0)$ from the first 
$(k+1)a+1$ columns plus the bottom element of column $(k+1)a+2$, then 
$(B',\sg')$, where $B'=(b_{2k+3}-1,b_{2k+4}, \ldots, b_t)$ 
and $\sg' = \mathrm{red}(\sg_{2(k+1)a+2} \ldots \sg_n$, will 
be a fixed point of $J_{\tau_a}$ of size $n-(2(k+1)a)-1$.  Note that 
in such a situation, we will have 
$\binom{n-(k+1)a-2}{(k+1)a}$ ways to choose the elements of 
that lie in the top rows of the first $(k+1)a-1$ columns of $A(O)$. 
Note that the powers of $y$ coming from the bricks $b_1, \ldots, b_{2k_2}$ is 
$y^{(k+1)a}$.
It follows 
that the elements in Subcase (2.B) contribute 
$$\sum_{k=0}^{\lfloor (n-2a-2)/(2a)\rfloor} \binom{n-(k+1)a-2}{(k+1)a} 
y^{(k+1)a}U_{\tau_a,n-(2(k+1)a)-1}(y)$$
to $U_{\tau_a,n}(y)$. \\
\ \\
Therefore, the recursion for the polynomials $U_{\tau,n}(y)$ is given by 
\begin{align*}
U_{\tau,n}(y)  = & ~  (1-y)U_{\tau,n-1}(y)  - 
\sum_{k=0}^{\lfloor (n-2a)/(2a)\rfloor} \binom{n-(k+1)a-1}{(k+1)a-1} y^{(k+1)a-1}
U_{\tau_a,n-(2(k+1)a)+1}(y) \\
& \qquad \quad +\sum_{k=0}^{\lfloor (n-2a-2)/(2a)\rfloor} \binom{n-(k+1)a-2}{(k+1)a} 
y^{(k+1)a} U_{\tau_a,n-(2(k+1)a)-1}(y).     
\end{align*}
This concludes the proof of Theorem \ref{thm:162534}. \qed




\begin{thebibliography}{20}


\bibitem{AAM} R.E.L. Aldred, M.D. Atkinson, and D.J. McCaughan, 
Avoiding consecutive patterns in permutations, \emph{Adv. in Applied Math.},
{\bf 45: Issue 3} (2010), 449-461. 


\bibitem{BR} Q.T.  Bach and J.B. Remmel, Generating functions for descents over 
permutations which avoid sets of consecutive patterns, Australian Journal 
of Combinatorics. {\bf 64} (2016), 194-231.




\bibitem{BR2} Q.T. Bach and J.B. Remmel, Descent c-Wilf equivalence, to appear 
in Discrete Mathematics and Theoretical Computer Science. 


\bibitem{B1} A.M. Baxter, Refining enumeration schemes to count according 
to inversion number, Pure Mathematics and Applications, {\bf 21 (2)} (2010),
 136-160. 



\bibitem{B2} A.M. Baxter, Refining enumeration schemes to count according 
to permutation statistics, Electronic J. Comb., {\bf 21: Issue 2} (2014). 



\bibitem{DK} V. Dotsenko and A. Khoroshkin, Anick-type resolutions and 
consecutive pattern avoidance, arXiv:1002.2761v1 (2010).




\bibitem{DR} A. Duane and J. Remmel, Minimal overlapping patterns in colored
permutations, Electronic J. Combinatorics, {\bf 18} (2) (2011).


\bibitem{Eg1} O. E{\u{g}}ecio{\u{g}}lu and J. B. Remmel, Brick tabloids and the connection matrices between bases of symmetric functions, {\em Discrete Appl. Math.}, \textbf{34} (1991), no.~1-3, 107--120, Combinatorics and theoretical computer science (Washington, DC, 1989).
  
\bibitem{EN}
S, Elizalde and M. Noy, Consecutive patterns in
permutations, Adv.
  in Appl. Math. \textbf{30} (2003), no.~1-2, 110-125, Formal power series and
  algebraic combinatorics (Scottsdale, AZ, 2001).

\bibitem{EN2} S. Elizalde and M. Noy, Clusters, generating 
functions and asymptotics for consecutive patterns in permutations,
Advances in App. Math., {\bf 49}(2012), 351-374. 

 




\bibitem{EKP} R. Ehrenborg, S. Kitaev, and P. Perry, A spectral approach 
to consecutive pattern-avoiding permutations, J. of Combinatorics, 
{\bf 2} (2011), 305-353. 



\bibitem{GJ} I.P. Goulden and D.M. Jackson, \emph{Combinatorial Enumeration},
A Wiley-Interscience Series in Discrete Mathematics, John Wiley \& Sons Inc,
New York, (1983).











\bibitem{JR1} M. Jones and J. B. Remmel, Pattern Matching in the Cycle Structures of Permutations, {\em Pure Math. and Applications}, {\bf 22} (2011), 
173-208.



\bibitem{JR} M. Jones and J. B. Remmel, A reciprocity approach to computing 
generating functions for permutations with no pattern matches, 
Discrete Mathematics and Theoretical Computer Science, 
DMTCS Proceedings, 23 International Conference on Formal 
Power Series and Algebraic Combinatorics (FPSAC 2011), {\bf 119} (2011), 
551-562.


\bibitem{JR2} M. Jones and J. Remmel, A reciprocity method for 
computing generating function over the set of permutations 
with no consecutive occurrences of $\tau$, 
Discrete Mathematics, {\bf 313} Issue 23 (2013), 2712-2729. 


\bibitem{JR3} M. Jones and J. Remmel, Generating functions 
for the number of permutations with no 
consecutive occurrences of  $1p23 \cdots (p-1)$ or $13\cdots(p-1)2p$,  
to appear in Pure Mathematics and Applications.



\bibitem{Kit1}
S. Kitaev, Partially ordered generalized patterns,
Discrete Math.{\bf 298} (2005), 212-229.


\bibitem{Kitbook} S. Kiteav, {\em Patterns in permutations and words}, 
Springer-Verlag, 2011. 



\bibitem{MenRem}  A. Mendes and J.B. Remmel, 
Permutations and words counted by 
consecutive patterns, Adv. Appl. Math, {\bf 37} 4, (2006), 443-480.

\bibitem{oeis} N.~J.~A.~Sloane, The on-line encyclopedia of integer sequences,
published electronically at \phantom{*} {\tt
http:/$\!\!$/www.research.att.com/\~{}njas/sequences/}.


\bibitem{Stanley} R.P. Stanley, {\em Enumerative Combinatorics, vol. 2}, 
Cambridge Studies in Advanced Mathematics 62, Cambridge University 
Press, (1999). 






\end{thebibliography}
\end{document}